\documentclass[11pt,a4paper]{article}
\usepackage[utf8]{inputenc}
\usepackage{amsmath,amsfonts,amssymb,amsthm,dsfont,enumitem,leftidx,tikz-cd}
\usepackage{tikz}
\usetikzlibrary{shapes}
\author{Rodrigo T. Sato Mart\'in de Almagro
	\\[2mm]
	{\small  Instituto de Ciencias Matem\'aticas (CSIC-UAM-UC3M-UCM)} \\
	{\small C/Nicol\'as Cabrera 13-15, 28049 Madrid, Spain}}
\title{Convergence of Lobatto-type Runge-Kutta methods for partitioned differential-algebraic systems of index 2}

\tikzset{
  solid node/.style={circle,draw,inner sep=1.2,fill=black},
  hollow node/.style={circle,draw,inner sep=1.2},
  triangle node/.style={regular polygon, regular polygon sides=3,draw,fill=black!20,inner sep=1.2},
}

\makeatletter
\def\maketag@@@#1{\hbox{\m@th\normalfont\normalsize#1}}
\makeatother

\newtheorem{theorem}{Theorem}[section]
\newtheorem{proposition}{Proposition}[section]
\newtheorem{corollary}{Corollary}[theorem]
\newtheorem{lemma}[theorem]{Lemma}

\theoremstyle{remark}
\newtheorem*{remark}{Remark}

\theoremstyle{definition}

\newcommand{\dsum}{\displaystyle\sum}

\begin{document}
\maketitle

\begin{abstract}
In this paper we propose a numerical scheme for partitioned systems of index 2 DAEs, such as those arising from nonholonomic mechanical problems and prove the order of a certain class of Runge-Kutta methods we call of Lobatto-type. The study of nonholonomic systems has  recently shown a new interest in that theory and also in its relation to the new developments in control theory, subriemannian geometry, robotics, etc. The proofs and general outline of the paper follow a similar procedure of the one by L.O. Jay  \cite{Jay93} in the non-partitioned setting, but we tackle the issue of having two different sets of coefficients in interaction.
\end{abstract}

\section{Introduction}
Let $N, M$ be smooth manifolds such that $M \subseteq N$. Assume that $\dim N = n$ and $\mathrm{codim} M = m$ and let $M$ be defined as the null-set of $\phi: N \to \mathbb{R}^m$. A generic explicit differential equation on $M$ can be recast into a semi-explicit index 2 differential algebraic equation (DAE) on $N$ taking the form:
\begin{equation}
\left\lbrace
\begin{array}{rl}
\dot{y} &= f(y,z)\\
0 &= \phi(y)
\end{array}\right.
\label{eq:general_system}
\end{equation}
where $y \in N$ and $z \in V$, with $V$ a vector space such that $\dim V = m$. Studies on the numerical solution of initial value problems (IVP) for such general systems on vector spaces can be found as part of the bibliography that serves as foundation for this paper, such as \cite{HaLuRo89} or \cite{Jay93}.

We are interested in a subset of such problems, which will be referred to as \emph{partitioned}, where $y = (q,p)$, $\dim Q = \dim P = n$ (thus in this case $\dim N = 2 n$), and $\lambda \in \mathbb{R}^m$.
\begin{equation}
\left\lbrace\begin{array}{rl}
\dot{q} &= f(q,p)\\
\dot{p} &= g(q,p,\lambda)\\
0 &= \phi(q,p)
\end{array}\right.
\label{eq:partitioned_system}
\end{equation}
Such is the case of the equations of motion of nonholonomic mechanical systems which motivates our study. Remember that nonholonomic equations are in Hamiltonian form 
	\begin{equation}
	\left\lbrace\begin{array}{rl}
	\dot{q}^i &={\displaystyle \vphantom{\sum_{j = 1}^N} \frac{\partial H}{\partial p_i}}\\
	\dot{p}_i &={\displaystyle \vphantom{\sum_{j = 1}^N} -\frac{\partial H}{\partial q^i}+\lambda_{\alpha} \mu^{\alpha}_i }\\
	0 &={ \displaystyle \vphantom{\sum_{j = 1}^N} \mu^{\alpha}_i\frac{\partial H}{\partial p_i} }
	\end{array}\right.
	\label{eq:partitioned_system-nonholonomic}
	\end{equation}
for a Hamiltonian function $H(q,p)$ and linear nonholonomic constraints $\mu_i^{\alpha}(q) \dot{q}^i=0$. An IVP for this partitioned DAE is defined by an initial condition $(q_0, p_0, \lambda_0) \in M \times \mathbb{R}^m$.

The development and application of the methods shown here in the case of nonholonomic mechanical systems will be the subject of a follow-up paper where numerical experiments will also be performed \cite{NonholonomicMartinSato18}.

For the remainder of the paper we will assume that $f$, $g$ and $\phi$ are sufficiently differentiable and that $\left(D_2 \phi D_3 g\right)(q,p,\lambda)$ remains invertible in a neighbourhood of the exact solution. Here $D_i$ means derivative with respect to the $i$-th argument.

\section{Lobatto-type methods}
A numerical solution of an IVP for \eqref{eq:general_system} can be found using a $s$-stage Runge-Kutta method with coefficients $(a_{i j}, b_{j})$. Writing the corresponding equations is a relatively trivial matter, taking the form:
\begin{subequations}
	\label{eq:RK_general_system}
	\begin{alignat}{2}
		y_1 &= y_0 + h \dsum_{j = 1}^s b_{j} k_j, \quad & z_1 &= z_0 + h \dsum_{j = 1}^s b_{j} l_j\\
		Y_i &= y_0 + h \dsum_{j = 1}^s a_{i j} k_j, \quad & Z_i &= z_0 + h \dsum_{j = 1}^s a_{i j} l_j\\
		k_i &= f(Y_i,Z_i), \quad & 0 &= g(Y_i) \vphantom{\dsum_{j = 1}^s}.
	\end{alignat}
\end{subequations}
Note that these $l_j$ are not given explicitly and must instead be solved for with the help from the constraint equations. In fact under some assumptions on the RK coefficients we may eliminate the equations for the $z$ and $Z$ variables completely.

Now, a numerical solution of an IVP for eqs.\eqref{eq:partitioned_system} can also be found using an $s$-stage \emph{partitioned} Runge-Kutta method but already the correct application of such a scheme is non-trivial. One could naively write:
\begin{subequations}
	\label{eq:RK_partitioned_system_naive}
	\begin{alignat}{3}
		q_1 &= q_0 + h \dsum_{j = 1}^s b_{j} V_j,\;\; & p_1 &= p_0 + h \dsum_{j = 1}^s \hat{b}_{j} W_j,\;\; & \lambda_1 &= \lambda_0 + h \dsum_{j = 1}^s \tilde{b}_j U_j,\\
		Q_i &= q_0 + h \dsum_{j = 1}^s a_{i j} V_j,\;\; & P_i &= p_0 + h \dsum_{j = 1}^s \hat{a}_{i j} W_j,\;\; & \Lambda_i &= \lambda_0 + h \dsum_{j = 1}^s \tilde{a}_{i j} U_j,\\
		V_i &= f(Q_i, P_i),\;\; & W_i &= g(Q_i, P_i, \Lambda_i),\;\; & 0 &= \phi(Q_i, P_i)\vphantom{\dsum_{j = 1}^s}.
	\end{alignat}
\end{subequations}
Again, $U_j$ are not given explicitly and, as above, in some cases, it may also be possible to eliminate the equations for $\lambda$ and $\Lambda$. Unfortunately such a system of equations may have certain issues, both from a solvability point of view and from a numerical convergence point of view. This is especially true for the particular case of partitioned Runge-Kutta methods that we will consider.

In \cite{Jay93} the author considers Runge-Kutta methods satisfying the hypotheses:
\begin{enumerate}[label=H\arabic*]
\item \label{itm:H1} $a_{1 j} = 0$ for $j = 1, ..., s$;
\item \label{itm:H2} the submatrix $\tilde{A} := (a_{i j})_{i,j \geq 2}$ is invertible;
\item \label{itm:H3} $a_{s j} = b_j$ for $j = 1, ..., s$ (the method is \emph{stiffly accurate}).
\end{enumerate}
\ref{itm:H1} implies that $c_1 =\sum_{j=1}^{s}a_{1j}= 0$ and for eqs. \eqref{eq:RK_general_system} $Y_1 = y_0$, $Z_1 = z_0$. \ref{itm:H3} implies that $y_1 = Y_s$, $z_1 = Z_s$. Furthermore, if the method is \emph{consistent}, i.e., $\sum_j b_j = 1$, then \ref{itm:H3} implies $c_s = 1$. For eqs. \eqref{eq:RK_partitioned_system_naive} if $(\tilde{a}_{i j},\tilde{b}_j)$ also satisfies the hypotheses, then $Q_1 = q_0$, $\Lambda_1 = \lambda_0$, $Q_s = q_1$ and $\Lambda_s = \lambda_1$. The most salient example of these methods is the \textbf{Lobatto IIIA}, which is a continuous collocation method.

The \textbf{Lobatto IIIB} is a family of discontinuous collocation methods which are symplectic conjugated to the IIIA methods. Two Runge-Kutta methods, $(a_{i j}, b_j)$ and $(\hat{a}_{i j}, \hat{b}_j)$, satisfying the \emph{compatibility condition} $\sum_{j = 1}^s \hat{a}_{i j}$ $= \hat{c_i} = c_i = \sum_{j = 1}^s a_{i j}$, are symplectic conjugated if they satisfy:
\begin{align}
&b_i \hat{a}_{i j} + \hat{b}_j a_{j i} = b_i \hat{b}_j \quad \text{for } i,j = 1,..., s\\
&b_j = \hat{b}_j \quad \text{for } j = 1,..., s
\end{align}
Together they form the \textbf{Lobatto IIIA-IIIB} family of \textbf{symplectic partitioned Runge-Kutta} methods which is precisely the one we want to study (see also \cite{NoWa81,HaLuWa2006}). 

Note that Lobatto IIIB methods do not satisfy any of the hypotheses above stated. In fact any symplectic conjugate method to a method satisfying those hypotheses must necessarily be such that:
\begin{enumerate}[label=H\arabic*']
\item \label{itm:H1'} $\hat{a}_{i s} = 0$ for $i = 1, ..., s$;
\item \label{itm:H2'} $\hat{a}_{i 1} = \hat{b}_1$ for $i = 1, ..., s$.
\end{enumerate}
Obviously the submatrix $\hat{\tilde{A}} := (\hat{a}_{i j})_{i,j \geq 2}$ is never invertible because of \ref{itm:H1'}, and this is the culprit of the solvability issues of \eqref{eq:RK_partitioned_system_naive}.

For such methods, we propose the following equations for the numerical solution of the partitioned IVP:

\begin{subequations}
	\label{eq:RK_partitioned_system}
	\begin{alignat}{2}
		q_1 &= q_0 + h \dsum_{j = 1}^s b_{j} V_j,\quad & p_1 &= p_0 + h \dsum_{j = 1}^s \hat{b}_{j} W_j,\\
		Q_i &= q_0 + h \dsum_{j = 1}^s a_{i j} V_j,\quad & P_i &= p_0 + h \dsum_{j = 1}^s \hat{a}_{i j} W_j,\\
		p_i &= p_0 + h \dsum_{j = 1}^s a_{i j} W_j,\quad & 0 &= \phi(Q_i, p_i),\\
		V_i &= f(Q_i, P_i),\quad & W_i &= g(Q_i, P_i, \Lambda_i) \vphantom{\dsum_{j = 1}^s}.
	\end{alignat}
\end{subequations}
together with $\Lambda_1 = \lambda_0$ and $\Lambda_s = \lambda_1$. It should be noted that, although similar, these methods do not generally coincide with the SPARK methods proposed by L. O. Jay in \cite{Jay09}.

There are several \emph{simplifying assumptions} that Runge-Kutta methods satisfy:
\begin{align*}
&B(p): \sum_{i = 1}^s b_i c_i^{k-1} = \frac{1}{k} \quad \text{for } k = 1,..., p\\
&C(q): \sum_{j = 1}^s a_{i j} c_j^{k-1} = \frac{c_i^k}{k} \quad \text{for } i = 1, ..., s,\; k = 1,..., q\\
&D(r): \sum_{i = 1}^s b_i c_i^{k-1} a_{i j} = \frac{b_j (1 - c_j^k)}{k} \quad \text{for } j = 1, ..., s,\; k = 1,..., r
\end{align*}

When referring to these assumptions for a Runge-Kutta method $(\hat{A},\hat{B})$ we will write them as $\hat{X}(\hat{y})$. Note that if $A$ and $\hat{A}$ are two symplectic conjugated methods, then $\hat{p} = p$, $C(q)$ implies $\hat{r} = q$, and conversely $D(r)$ implies $\hat{q} = r$.

Apart from these, there are a few more simplifying assumptions that pairs of compatible methods satisfy (see \cite{Jay96}):
\begin{align*}
&C\hat{C}(Q): \sum_{j = 1}^s \sum_{l = 1}^s a_{i j} \hat{a}_{j l} c_l^{k-2} = \frac{c_i^k}{k (k - 1)} \quad \text{for } i = 1, ..., s,\; k = 2,..., Q\\
&D\hat{D}(R): \sum_{i = 1}^s \sum_{j = 1}^s b_i c_i^{k-2} a_{i j} \hat{a}_{j l} = \frac{b_l}{k (k - 1)} \left[(k - 1) - (k c_l - c_l^k)\right]
&\tag*{for  $l = 1, ..., s$, $k = 2,..., R$\hspace{1.2cm}}
\end{align*}
\begin{align*}
&\hat{C} C(\hat{Q}): \sum_{j = 1}^s \sum_{l = 1}^s \hat{a}_{i j} a_{j l} c_l^{k-2} = \frac{c_i^k}{k (k - 1)} \quad \text{for } i = 1, ..., s,\; k = 2,..., \hat{Q}\\
&\hat{D} D(\hat{R}): \sum_{i = 1}^s \sum_{j = 1}^s \hat{b}_i c_i^{k-2} \hat{a}_{i j} a_{j l} = \frac{\hat{b}_l}{k (k - 1)} \left[(k - 1) - (k c_l - c_l^k)\right]
&\tag*{for  $l = 1, ..., s$, $k = 2,..., \hat{R}$\hspace{1.2cm}}
\end{align*}

It can be shown that if both methods are symplectic conjugated then, $Q = R = p - r$ and $\hat{Q} = \hat{R} = p - q$. In particular, Lobatto methods satisfy $B(2 s - 2)$, $C(s)$, $D(s-2)$, $\hat{B}(2 s - 2)$, $\hat{C}(s-2)$, $\hat{D}(s)$, as well as $C\hat{C}(s), D\hat{D}(s), \hat{C}C(s-2), \hat{D}D(s-2)$.

Before moving on there is a function associated to a Runge-Kutta method that we need to define. Consider the linear problem $\dot{y} = \lambda y$, and apply one step of the given method for an initial value $y_0$. The function $R(z)$ defined by $y_1 = R(h \lambda) y_0$ is the so-called stability function of the method.

For an arbitrary Runge-Kutta method we have that
\begin{equation*}
R(z) = 1 + z b (\mathrm{Id} - z A)^{-1} \mathds{1},
\end{equation*}
where $A = \left(a_{i j}\right)$, $b = (b_1,...,b_s)$ and $\mathds{1} = (1,...,1)^T$. In the particular case of a method satisfying hypothesis \ref{itm:H3} this can be reduced to:
\begin{equation*}
R(z) = e_s^T (\mathrm{Id} - z A)^{-1} \mathds{1},
\end{equation*}
where $e_i$ denotes an $s$-dimensional column vector whose entries are all zero except for its $i$-th entry which is 1.

%

\section{Existence, uniqueness and influence of perturbations}
\begin{theorem}
\label{thm:existence_and_uniqueness_of_solution}

Let $U \subset N \times \mathbb{R}^m$ be a fixed neighbourhood of $(q_0, p_0, \lambda_0) = (q_0(h), p_0(h), \lambda_0(h))$, a set of $h$-dependent starting values, and assume:
\begin{align*}
&\phi(q_0, p_0) = 0\\
&(D_1 \phi \cdot f)(q_0,p_0) + (D_2 \phi \cdot g)(q_0,p_0,\lambda_0) = \mathcal{O}(h)\\
&(D_2 \phi \cdot D_3 g)(q,p,\lambda) \text{ invertible in }U.
\end{align*}
Assume also that the Runge-Kutta coefficients $A$ verify the hypotheses \ref{itm:H1} and \ref{itm:H2}, and that $\hat{A}$ is compatible with the first and satisfies \ref{itm:H1'}. Then for $h \leq h_0$ there exists a locally unique solution to:
\begin{subequations}
	\label{eq:thm_existence_RK_subsystem}
	\begin{alignat}{1}
		Q_i &= q_0 + h \dsum_{j = 1}^s a_{i j} f(Q_j,P_j),\\
		p_i &= p_0 + h \dsum_{j = 1}^s a_{i j} g(Q_j,P_j,\Lambda_j),\\
		P_i &= p_0 + h \dsum_{j = 1}^s \hat{a}_{i j} g(Q_j,P_j,\Lambda_j),\\
		0 &= \phi(Q_i,p_i) \vphantom{\dsum_{j = 1}^s},
	\end{alignat}
\end{subequations}
with $\Lambda_1 = \lambda_0$, satisfying:
\begin{equation*}
\begin{array}{rl}
Q_i - q_0 &= \mathcal{O}(h)\\
p_i - p_0 &= \mathcal{O}(h)\\
P_i - p_0 &= \mathcal{O}(h)\\
\Lambda_i - \lambda_0 &= \mathcal{O}(h)
\end{array}
\end{equation*}
\end{theorem}

\begin{proof}
The proof of existence differs little from what is already offered in \cite{HaLuRo89} (for invertible $A$ matrix) or \cite{Jay93} (for $A$ satisfying the hypotheses \ref{itm:H1} and \ref{itm:H2}). The idea is to consider a homotopic deformation of the equations which leads to a system of differential equations where the existence of a solution for the corresponding IVP implies the existence of a solution to the original system.

The proposed homotopy is:
\begin{equation*}
\begin{array}{rl}
Q_i &= q_0 + h \dsum_{j = 1}^s a_{i j} \left[ f(Q_j,P_j) + (\tau - 1) f(q_0,p_0)\right]\\
p_i &= p_0 + h \dsum_{j = 1}^s a_{i j} \left[ g(Q_j,P_j,\Lambda_j) + (\tau - 1) g(q_0,p_0,\lambda_0)\right]\\
P_i &= p_0 + h \dsum_{j = 1}^s \hat{a}_{i j} \left[ g(Q_j,P_j,\Lambda_j) + (\tau - 1) g(q_0,p_0,\lambda_0)\right]\\
0 &= \phi(Q_i,p_i) + (\tau - 1) \phi(q_0,p_0) \vphantom{\dsum_{j = 1}^s}
\end{array}
\end{equation*}

The main differences in the proof lie in the complementary relation between the equations for $p_i$ and $P_i$. One needs to consider the differential system obtained by derivation with respect to the homotopy parameter $\tau$. The resulting system takes the form:
\begin{subequations}
	\begin{align}
	\dot{Q}_i &= h \sum_{j = 1}^s a_{i j} \left[ D_1 f(Q_j,P_j) \dot{Q}_j + D_2 f(Q_j,P_j) \dot{P}_j + f(q_0,p_0)\right] \label{eq:thm_dot_Q}\\
	\dot{p}_i &= h \sum_{j = 1}^s a_{i j} \left[ D_1 g(Q_j,P_j,\Lambda_j) \dot{Q}_j + D_2 g(Q_j,P_j,\Lambda_j) \dot{P}_j \right.\nonumber\\
	&\left. + D_3 g(Q_j,P_j,\Lambda_j) \dot{\Lambda}_j + g(q_0,p_0,\lambda_0)\right]\label{eq:thm_dot_p}\\
	\dot{P}_i &= h \sum_{j = 1}^s \hat{a}_{i j} \left[ D_1 g(Q_j,P_j,\Lambda_j) \dot{Q}_j + D_2 g(Q_j,P_j,\Lambda_j) \dot{P}_j \right.\nonumber\\
	&\left. + D_3 g(Q_j,P_j,\Lambda_j) \dot{\Lambda}_j + g(q_0,p_0,\lambda_0)\right]\label{eq:thm_dot_P}\\
	0 &= D_1\phi(Q_i,p_i)\dot{Q}_i + D_2\phi(Q_i,p_i)\dot{p}_i +  \phi(q_0,p_0)\label{eq:thm_constraint}
	\end{align}
\end{subequations}

Note that $\dot{p}_i$ depends on $\dot{Q}_j, \dot{P}_j, \dot{\Lambda}_j$, but not on $\dot{p}_j$. In fact $\dot{p}_i$ only appears in the equations for $\phi$, where it prevents the entrance of $\hat{a}$ terms. Thus the differential system that must be solved can be reduced to the $\dot{P}_j$, $\forall j = 1,...,s$ and $\dot{Q}_j, \dot{\Lambda}_j$, $\forall j = 2,...,s$ variables. The rest of the proof follows closely what the other authors do.

A remark worth mentioning is that the key of the remainder of the proof is the use of the invertibility of $D_2 \phi (\tilde{A} \otimes I) D_3 g$, which is a term arising from eq. \eqref{eq:thm_constraint}. As stated in the former section, if the system were described by eq.
 \eqref{eq:RK_partitioned_system_naive} we would instead have $D_2 \phi (\hat{\tilde{A}} \otimes I) D_3 g$, which is not invertible by \ref{itm:H1'}, rendering the system unsolvable.

The proof of uniqueness remains essentially the same.
\end{proof}


\begin{theorem}
\label{thm:impact_of_perturbations}
Under the assumptions of theorem \ref{thm:existence_and_uniqueness_of_solution}, let $Q_i$, $p_i$, $P_i$, $\Lambda_i$ be the solution of the system in said theorem. Now consider the perturbed values $\hat{Q}_i$, $\hat{p}_i$, $\hat{P}_i$, $\hat{\Lambda}_i$ satisfying:
\begin{subequations}
	\label{eq:thm_perturbed_system}
	\begin{alignat}{1}
		\hat{Q}_i &= \hat{q}_0 + h \dsum_{j = 1}^s a_{i j} f(\hat{Q}_j,\hat{P}_j) + h\delta_{Q,i}\\
		\hat{p}_i &= \hat{p}_0 + h \dsum_{j = 1}^s a_{i j} g(\hat{Q}_j,\hat{P}_j,\hat{\Lambda}_j) + h\delta_{p,i}\\
		\hat{P}_i &= \hat{p}_0 + h \dsum_{j = 1}^s \hat{a}_{i j} g(\hat{Q}_j,\hat{P}_j,\hat{\Lambda}_j) + h\delta_{P,i}\\
		0 &= \phi(\hat{Q}_i,\hat{p}_i) + \theta_i \vphantom{\dsum_{j = 1}^s}
	\end{alignat}
\end{subequations}
with $\hat{\Lambda}_1 = \hat{\lambda}_0$. Additionally, assume that:
\begin{equation}
\begin{array}{rl}
\hat{q}_0 - q_0 &= \mathcal{O}(h)\\
\hat{p}_0 - p_0 &= \mathcal{O}(h)\\
\delta_i &= \mathcal{O}(h)\\
\theta_i &= \mathcal{O}(h^2)
\end{array}
\label{eq:thm_perturbed_distance_to_initial_values}
\end{equation}

Then, using the notation $\Delta X := \hat{X} - X$ and $\left\Vert X \right\Vert := \max_i \left\Vert X_i \right\Vert$, for small $h$ we have:
\begin{small}
\begin{align*}
\left\Vert \Delta Q_i \right\Vert &\leq C \left( \left\Vert \Delta q_0 \right\Vert + h \left\Vert \Delta p_0 \right\Vert + h^2 \left\Vert \Delta \lambda_0 \right\Vert + h \left\Vert \delta_Q \right\Vert + h^2 \left\Vert \delta_p \right\Vert + h^2 \left\Vert \delta_P \right\Vert + h \left\Vert \theta \right\Vert \right)\\
\left\Vert \Delta p_i \right\Vert &\leq C \left( \left\Vert \Delta q_0 \right\Vert + \left\Vert \Delta p_0 \right\Vert + h^2 \left\Vert \Delta \lambda_0 \right\Vert + h^2 \left\Vert \delta_Q \right\Vert + h \left\Vert \delta_p \right\Vert + h^2 \left\Vert \delta_P \right\Vert + \left\Vert \theta \right\Vert \right)\\
\left\Vert \Delta P_i \right\Vert &\leq C \left( \left\Vert \Delta q_0 \right\Vert + \left\Vert \Delta p_0 \right\Vert + h \left\Vert \Delta \lambda_0 \right\Vert + h^2 \left\Vert \delta_Q \right\Vert + h \left\Vert \delta_p \right\Vert + h \left\Vert \delta_P \right\Vert + \left\Vert \theta \right\Vert \right)\\
\left\Vert \Delta \Lambda_i \right\Vert &\leq \frac{C}{h} \left( h \left\Vert \Delta q_0 \right\Vert + h \left\Vert \Delta p_0 \right\Vert + h \left\Vert \Delta \lambda_0 \right\Vert + h \left\Vert \delta_Q \right\Vert + h \left\Vert \delta_p \right\Vert + \left\Vert \theta \right\Vert \right)
\end{align*}
\end{small}

\end{theorem}

\begin{proof}
To tackle this problem we first subtract eq. \eqref{eq:thm_existence_RK_subsystem} from eq. \eqref{eq:thm_perturbed_system} and linearise, obtaining:
\begin{align*}
\Delta Q_i &= \Delta q_0 + h \sum_{j = 1}^s a_{i j} \left[ D_1 f(Q_j,P_j) \Delta Q_j + D_2 f(Q_j,P_j) \Delta P_j\right] + h \delta_{Q,i}\\
&+ \mathcal{O}\left(h \Vert\Delta Q\Vert^2 + h\Vert\Delta P\Vert^2 + h \Vert\Delta Q\Vert \Vert\Delta P\Vert\right)\\
\Delta p_i &= \Delta p_0 + h \sum_{j = 1}^s a_{i j} \left[ D_1 g(Q_j,P_j,\Lambda_j) \Delta Q_j + D_2 g(Q_j,P_j,\Lambda_j) \Delta P_j \right.\\
&\left. + D_3 g(Q_j,P_j,\Lambda_j) \Delta \Lambda_j\right] + h\delta_{p,i} + \mathcal{O}\left(h \Vert\Delta Q\Vert^2 + h\Vert\Delta P\Vert^2\right.\\
&\left.+ h \Vert\Delta Q\Vert \Vert\Delta P\Vert + h \Vert\Delta Q\Vert \Vert\Delta \Lambda\Vert + h \Vert\Delta P\Vert \Vert\Delta \Lambda\Vert\right)\\
\Delta P_i &= \Delta p_0 + h \sum_{j = 1}^s \hat{a}_{i j} \left[ D_1 g(Q_j,P_j,\Lambda_j) \Delta Q_j + D_2 g(Q_j,P_j,\Lambda_j) \Delta P_j \right.\\
&\left. + D_3 g(Q_j,P_j,\Lambda_j) \Delta \Lambda_j\right] + h\delta_{P,i} + \mathcal{O}\left(h \Vert\Delta Q\Vert^2 + h\Vert\Delta P\Vert^2\right.\\
&\left.+ h \Vert\Delta Q\Vert \Vert\Delta P\Vert + h \Vert\Delta Q\Vert \Vert\Delta \Lambda\Vert + h \Vert\Delta P\Vert \Vert\Delta \Lambda\Vert\right)\\
0 &= D_1\phi(Q_i,p_i) \Delta Q_i + D_2\phi(Q_i,p_i) \Delta p_i + \theta_i\\
&+ \mathcal{O}\left(\Vert\Delta Q\Vert^2 + \Vert\Delta p\Vert^2 + \Vert\Delta Q\Vert \Vert\Delta p\Vert\right)
\end{align*}

We will write this system of equations as separate matrix subsystems:
\begin{small}
\begin{align*}
&\left[\begin{smallmatrix}
      \Delta Q_1\\
\Delta \tilde{Q}\\
      \Delta P_1\\
\Delta \tilde{P}\\
\end{smallmatrix}
\right] = \left[\begin{smallmatrix}
                \Delta q_0\\
I_{s-1} \otimes \Delta q_0\\
                \Delta p_0\\
I_{s-1} \otimes \Delta p_0
\end{smallmatrix}
\right] + h\left[\begin{smallmatrix}
      \delta_{Q,1}\\
\tilde{\delta}_{Q}\\
      \delta_{P,1}\\
\tilde{\delta}_{P}\\
\end{smallmatrix}
\right] + h\left[\begin{smallmatrix}
      A_1^1 \otimes I_{n} &       A^1 \otimes I_{n} &                               0 &                             0\\
\tilde{A}_1 \otimes I_{n} & \tilde{A} \otimes I_{n} &                               0 &                             0\\
                        0 &                       0 &       \hat{A}_1^1 \otimes I_{n} &       \hat{A}^1 \otimes I_{n}\\
                        0 &                       0 & \hat{\tilde{A}}_1 \otimes I_{n} & \hat{\tilde{A}} \otimes I_{n}
\end{smallmatrix}
\right]\nonumber\\
& \times \left(\left[\begin{smallmatrix}
D_1 f_1 &             0 & D_2 f_1 &             0\\
      0 & D_1 \tilde{f} &       0 & D_2 \tilde{f}\\
D_1 g_1 &             0 & D_2 g_1 &             0\\
      0 & D_1 \tilde{g} &       0 & D_2 \tilde{g}
\end{smallmatrix}
\right] \left[\begin{smallmatrix}
      \Delta Q_1\\
\Delta \tilde{Q}\\
      \Delta P_1\\
\Delta \tilde{P}
\end{smallmatrix}
\right] + \left[\begin{smallmatrix}
      0 &             0\\
      0 &             0\\
D_3 g_1 &             0\\
      0 & D_3 \tilde{g}
\end{smallmatrix}
\right] \left[\begin{smallmatrix}
      \Delta \Lambda_1\\
\Delta \tilde{\Lambda}\\
\end{smallmatrix}
\right] \right)
\end{align*}
\begin{align*}
&\left[\begin{smallmatrix}
      \Delta Q_1\\
\Delta \tilde{Q}\\
      \Delta p_1\\
\Delta \tilde{p}\\
\end{smallmatrix}
\right] = \left[\begin{smallmatrix}
                \Delta q_0\\
I_{s-1} \otimes \Delta q_0\\
                \Delta p_0\\
I_{s-1} \otimes \Delta p_0
\end{smallmatrix}
\right] + h\left[\begin{smallmatrix}
      \delta_{Q,1}\\
\tilde{\delta}_{Q}\\
      \delta_{p,1}\\
\tilde{\delta}_{p}\\
\end{smallmatrix}
\right] + h \left[\begin{smallmatrix}
      A_1^1 \otimes I_{n} &       A^1 \otimes I_{n} &                         0 &                       0\\
\tilde{A}_1 \otimes I_{n} & \tilde{A} \otimes I_{n} &                         0 &                       0\\
                        0 &                       0 &       A_1^1 \otimes I_{n} &       A^1 \otimes I_{n}\\
                        0 &                       0 & \tilde{A}_1 \otimes I_{n} & \tilde{A} \otimes I_{n}
\end{smallmatrix}
\right]\nonumber\\
& \times \left(\left[\begin{smallmatrix}
D_1 f_1 &             0 & D_2 f_1 &             0\\
      0 & D_1 \tilde{f} &       0 & D_2 \tilde{f}\\
D_1 g_1 &             0 & D_2 g_1 &             0\\
      0 & D_1 \tilde{g} &       0 & D_2 \tilde{g}
\end{smallmatrix}
\right] \left[\begin{smallmatrix}
      \Delta Q_1\\
\Delta \tilde{Q}\\
      \Delta P_1\\
\Delta \tilde{P}
\end{smallmatrix}
\right] + \left[\begin{smallmatrix}
      0 &             0\\
      0 &             0\\
D_3 g_1 &             0\\
      0 & D_3 \tilde{g}
\end{smallmatrix}
\right] \left[\begin{smallmatrix}
      \Delta \Lambda_1\\
\Delta \tilde{\Lambda}\\
\end{smallmatrix}
\right] \right)
\end{align*}
\begin{equation*}
\left[\begin{smallmatrix}
D_1 \phi_1 &                0 & D_2 \phi_1 &                0\\
         0 & D_1 \tilde{\phi} &          0 & D_2 \tilde{\phi}\\
\end{smallmatrix}
\right] \left[\begin{smallmatrix}
      \Delta Q_1\\
\Delta \tilde{Q}\\
      \Delta p_1\\
\Delta \tilde{p}\\
\end{smallmatrix}
\right] + \left[\begin{smallmatrix}
     \theta_1\\
\tilde{\theta}
\end{smallmatrix}
\right] = 0
\end{equation*}
\end{small}

Let us rewrite this in shorthand notation as:
\begin{subequations}
	\label{eq:proof_perturbation_shorthand}
	\begin{alignat}{1}
		\Delta Y &= \Delta \eta + h \delta_Y + h \left(\leftidx{^A}{\phantom{.}}{_{\hat{A}}}\right) \left( D_y F \Delta Y + D_z F \Delta \Lambda \right)\label{eq:proof_perturbation_shorthand_Y}\\
		\Delta y &= \Delta \eta + h \delta_y + h \left(\leftidx{^A}{\phantom{.}}{_{A}}\right) \left( D_y F \Delta Y + D_z F \Delta \Lambda \right)\label{eq:proof_perturbation_shorthand_y}\\
		0 &= D_y \phi \Delta y + \theta \vphantom{\left(\leftidx{^A}{\phantom{.}}{_{A}}\right)}\label{eq:proof_perturbation_shorthand_constraint}
	\end{alignat}
\end{subequations}

Using hypothesis \ref{itm:H1} we find that:
\begin{align*}
\Delta Q_1 &= \Delta q_0 + h \delta_{Q,1}\\
\Delta p_1 &= \Delta p_0 + h \delta_{p,1}\\
D_1\phi(q_0,p_0) \Delta q_0 + D_2\phi(q_0,p_0) \Delta p_0 &= \mathcal{O}\left(h \Vert\delta_{Q,1}\Vert + h \Vert\delta_{p,1}\Vert + \Vert\theta_1\Vert\right.\\
&\left.+ \Vert\Delta q_0\Vert^2 + \Vert\Delta p_0\Vert^2 + \Vert\Delta q_0\Vert\Vert\Delta p_0\Vert \right)
\end{align*}

The equivalent version of the last line in \cite{Jay93} (formula 4.8) contains an erratum. It should read:
\begin{equation*}
g_y(\eta) \Delta\eta = \mathcal{O}(\Vert\Delta\eta\Vert^2 + h\Vert\delta_1\Vert + \Vert\theta_1\Vert)
\end{equation*}

Most of the proof will follow the lines of the one of \cite{Jay93}. We will first insert eq. \eqref{eq:proof_perturbation_shorthand_y} in the constraint eq. \eqref{eq:proof_perturbation_shorthand_constraint}
\begin{equation*}
D_y \phi \left[\Delta \eta + h \delta_y + h \left(\leftidx{^A}{\phantom{.}}{_{A}}\right) \left( D_y F \Delta Y + D_z F \Delta \Lambda \right)\right] + \theta = 0
\end{equation*}

Our mission will be to solve for $\Delta \Lambda$, but due to the singularity of $A$ it will not be possible to solve for the entire vector. Instead, abusing our notation a bit, we will separate the term as $D_z F \Delta \Lambda = D_z F_1 \Delta \Lambda_1 + D_z \tilde{F} \Delta \tilde{\Lambda}$, which leads to:
\begin{align*}
- h D_y \phi \left(\leftidx{^A}{\phantom{.}}{_{A}}\right) D_z \tilde{F} \Delta \tilde{\Lambda} = &D_y \phi \left[\Delta \eta + h \delta_y + h \left(\leftidx{^A}{\phantom{.}}{_{A}}\right) \left( D_y F \Delta Y + D_z F_1 \Delta \Lambda_1 \right)\right]\nonumber\\
&+ \theta
\end{align*}

Using \ref{itm:H1} and taking into account all the zeros that appear in the rest of the elements, the left-hand side can be reduced to $- h D_2 \tilde{\phi} \left(\tilde{A} \otimes I_{n}\right) D_3 \tilde{g} \Delta \tilde{\Lambda}$. Solving for $h \Delta \tilde{\Lambda}$ we get in matrix notation:
\begin{small}
\begin{align*}
&h \Delta \tilde{\Lambda}\\
=& - \left[\begin{smallmatrix}
         \tilde{0}_1 & \left(D_2 \tilde{\phi} \left(\tilde{A} \otimes I_{n}\right) D_3 \tilde{g}\right)^{-1}
\end{smallmatrix}
\right] \left\lbrace\left[\begin{smallmatrix}
     \theta_1\\
\tilde{\theta}
\end{smallmatrix}
\right] + \left[\begin{smallmatrix}
D_1 \phi_1 &                0 & D_2 \phi_1 &                0\\
         0 & D_1 \tilde{\phi} &          0 & D_2 \tilde{\phi}\\
\end{smallmatrix}
\right]\left[\left[\begin{smallmatrix}
                \Delta q_0\\
I_{s-1} \otimes \Delta q_0\\
                \Delta p_0\\
I_{s-1} \otimes \Delta p_0
\end{smallmatrix}
\right] + h\left[\begin{smallmatrix}
      \delta_{Q,1}\\
\tilde{\delta}_{Q}\\
      \delta_{p,1}\\
\tilde{\delta}_{p}\\
\end{smallmatrix}
\right]\right.\right.\nonumber\\
&+\left.\left. h \left[\begin{smallmatrix}
A_1^1 \otimes I_{n} &       A^1 \otimes I_{n} &                   0 &                       0\\
  A_1 \otimes I_{n} & \tilde{A} \otimes I_{n} &                   0 &                       0\\
                  0 &                       0 & A_1^1 \otimes I_{n} &       A^1 \otimes I_{n}\\
                  0 &                       0 &   A_1 \otimes I_{n} & \tilde{A} \otimes I_{n}
\end{smallmatrix}
\right] \left(\left[\begin{smallmatrix}
D_1 f_1 &             0 & D_2 f_1 &             0\\
      0 & D_1 \tilde{f} &       0 & D_2 \tilde{f}\\
D_1 g_1 &             0 & D_2 g_1 &             0\\
      0 & D_1 \tilde{g} &       0 & D_2 \tilde{g}
\end{smallmatrix}
\right] \left[\begin{smallmatrix}
      \Delta Q_1\\
\Delta \tilde{Q}\\
      \Delta P_1\\
\Delta \tilde{P}
\end{smallmatrix}
\right] + \left[\begin{smallmatrix}
      0\\
      0\\
D_3 g_1\\
      0
\end{smallmatrix}
\right] \Delta \Lambda_1 \right)\right]\right\rbrace\nonumber
\end{align*}
\end{small}

Let us introduce the notation:
\begin{align*}
\left(D_2 \phi A D_3 g\right)^{-} =& \left[\begin{smallmatrix}
     0_1^1 &                                                                                   0^1\\
\tilde{0}_1& \left(D_2 \tilde{\phi} \left(\tilde{A} \otimes I_{n}\right) D_3 \tilde{g}\right)^{-1}
\end{smallmatrix}
\right]
\end{align*}

We can now insert this back into $\Delta Y$ and obtain:
\begin{align}
\Delta Y =& \Delta \eta + h \delta_Y + h \left(\leftidx{^A}{\phantom{.}}{_{\hat{A}}}\right) \left( D_y F \Delta Y + D_z F_1 \Delta \Lambda_1 \right)\label{eq:proof_perturbation_Y_extended}\\
&- \left(\leftidx{^A}{\phantom{.}}{_{\hat{A}}}\right) D_z \tilde{F} \left(\leftidx{^0}{\phantom{.}}{_{\left(D_2 \phi A D_3 g\right)^{-}}}\right)\nonumber\\
&\times \left\lbrace D_y \phi  \left[ \Delta \eta + h \delta_y + h \left(\leftidx{^A}{\phantom{.}}{_{A}}\right) \left( D_y F \Delta Y + D_z F_1 \Delta \Lambda_1 \right)\right] + \theta\right\rbrace\nonumber
\end{align}

Introducing the projectors:
\begin{align*}
\Pi_{\hat{A}}^{A} 
:=& I - \left(\leftidx{^A}{\phantom{.}}{_{\hat{A}}}\right) D_z \tilde{F} \left(\leftidx{^0}{\phantom{.}}{_{\left(D_2 \phi A D_3 g\right)^{-}}}\right) D_y \phi\\
P_{A} &:= I - D_z \tilde{F} \left(\leftidx{^0}{\phantom{.}}{_{\left(D_2 \phi A D_3 g\right)^{-}}}\right) D_y \phi \left(\leftidx{^A}{\phantom{.}}{_{A}}\right)
\end{align*}
this expression can be further simplified as:
\begin{align}
\Delta Y =& \Pi_{\hat{A}}^{A} \left( \Delta \eta + h \delta_y \right) + h \left(\leftidx{^A}{\phantom{.}}{_{\hat{A}}}\right) P_{A} \left( D_y F \Delta Y + D_z F_1 \Delta \Lambda_1 \right)\label{eq:proof_perturbation_Y}\\
&- \left(\leftidx{^A}{\phantom{.}}{_{\hat{A}}}\right) D_z \tilde{F} \left(\leftidx{^0}{\phantom{.}}{_{\left(D_2 \phi A D_3 g\right)^{-}}}\right)\theta + h \left(\delta_Y - \delta_y\right)\nonumber
\end{align}

As for $\Delta y$, we have:
\begin{align}
\Delta y = &\Delta \eta + h \delta_y + h \left(\leftidx{^A}{\phantom{.}}{_{A}}\right) \left( D_y F \Delta Y + D_z F_1 \Delta \Lambda_1 \right)\label{eq:proof_perturbation_y_extended}\\
&- \left(\leftidx{^A}{\phantom{.}}{_{A}}\right) D_z \tilde{F} \left(\leftidx{^0}{\phantom{.}}{_{\left(D_2 \phi A D_3 g\right)^{-}}}\right)\nonumber\\
&\times \left\lbrace D_y \phi  \left[ \Delta \eta + h \delta_y + h \left(\leftidx{^A}{\phantom{.}}{_{A}}\right) \left( D_y F \Delta Y + D_z F_1 \Delta \Lambda_1 \right)\right] + \theta\right\rbrace\nonumber
\end{align}
which, using $\Pi_A := \Pi_{A}^{A}$, can be simplified as:
\begin{align}
\Delta y = & \Pi_A \left[\Delta \eta + h \delta_y + h \left(\leftidx{^A}{\phantom{.}}{_{A}}\right) \left( D_y F \Delta Y + D_z F_1 \Delta \Lambda_1 \right)\right]\label{eq:proof_perturbation_y}\\
&- \left(\leftidx{^A}{\phantom{.}}{_{A}}\right) D_z \tilde{F} \left(\leftidx{^0}{\phantom{.}}{_{\left(D_2 \phi A D_3 g\right)^{-}}}\right)\theta.\nonumber
\end{align}
From eqs.\eqref{eq:proof_perturbation_Y} and \eqref{eq:proof_perturbation_y} we can derive the result of the theorem almost directly. The trickiest term, $h^2 \left\Vert \Delta \lambda_0 \right\Vert$ in $\left\Vert \Delta p_i \right\Vert$, is the one already derived by Jay in \cite{Jay93}. Reading off the terms directly seems to point towards $h \left\Vert \Delta \lambda_0 \right\Vert$, but this estimation can be refined as follows. Realise that:
\begin{align}
\Pi_{A} =& \left[\begin{smallmatrix}
1 \otimes I_{n} &                                                  0 &               0 &                     0\\
              0 &                              I_{s-1} \otimes I_{n} &               0 &                     0\\
              0 &                                                  0 & 1 \otimes I_{n} &                     0\\
              0 & - \left(\tilde{A} \otimes I_{n}\right) \tilde{X}_1 &               0 & I_{s-1} \otimes I_{n} - \left(\tilde{A} \otimes I_{n}\right)\tilde{X}_2
\end{smallmatrix}
\right]\label{eq:proof_RK_total_PiA}\\
=& \left[\begin{smallmatrix}
1 \otimes I_{n} &                     0 &               0 &                     0\\
              0 & I_{s-1} \otimes I_{n} &               0 &                     0\\
              0 &                     0 & 1 \otimes I_{n} &                     0\\
              0 &     \tilde{\Pi}_{1,A} &               0 &     \tilde{\Pi}_{2,A}
\end{smallmatrix}
\right]\nonumber
\end{align}
with:
\begin{align}
\tilde{X}_i :=& D_3 \tilde{g} \left( D_2 \tilde{\phi} \left(\tilde{A} \otimes I_{n}\right) D_3 \tilde{g} \right)^{-1} D_i \tilde{\phi}\label{eq:def_X_i}\\
= & D_3 \tilde{g} \left(\tilde{A} \otimes I_{m}\right)^{-1} \left( D_2 \tilde{\phi} \left(\tilde{A} \otimes I_{n}\right) D_3 \tilde{g} \left(\tilde{A} \otimes I_{m}\right)^{-1}\right)^{-1} D_i \tilde{\phi}\nonumber
\end{align}
where in the second line we have inserted the identity matrix as $I_s = \tilde{A}^{-1} \tilde{A}$.

One can easily check that $\tilde{\Pi}_{2,A} \left(\tilde{A} \otimes I_{n}\right) D_3 \tilde{g} \left(\tilde{A} \otimes I_{m}\right)^{-1} = 0$. Now, the non-zero components of $h \Pi_A \left(\leftidx{^A}{\phantom{.}}{_{A}}\right) D_z F_1 \Delta \Lambda_1$ are $h \tilde{\Pi}_{2,A} D_3 g_1 \left(\tilde{A}_1 \otimes \Delta \Lambda_1\right)$, thus we can finally write:
\begin{align}
&h \tilde{\Pi}_{2,A} \left( \left(\tilde{A} \otimes I_{n}\right) D_3 \tilde{g} \left(\tilde{A} \otimes I_{n}\right)^{-1} - D_3 g_1\right) \left(\tilde{A}_1 \otimes \Delta \Lambda_1\right)\label{eq:proof_perturbation_Lambda_error}\\
&\quad\quad = h \tilde{\Pi}_{2,A} \mathcal{O}(h) \left(\tilde{A}_1 \otimes \Delta \Lambda_1\right)\nonumber\\
&\quad\quad = \mathcal{O}\left(h^2 \left\Vert\Delta \lambda_0\right\Vert\right)\nonumber
\end{align}
This cannot be done for $\left\Vert \Delta P_i \right\Vert$, which makes it $\mathcal{O}\left(h \left\Vert\Delta \lambda_0\right\Vert\right)$. Inserting this back into either $\Delta Y$ or $\Delta y$ confirms that $\left\Vert \Delta Q_i \right\Vert$ is $\mathcal{O}\left(h^2 \left\Vert\Delta \lambda_0\right\Vert\right)$.
\end{proof}

\begin{lemma}
\label{lem:RK_delta_error}
In addition to the hypotheses of theorem \ref{thm:existence_and_uniqueness_of_solution}, suppose that $C(q)$, $\hat{C}(\hat{q})$ and $C\hat{C}(Q)$ and that $(D_1 \phi \cdot f)(q_0,p_0) + (D_2 \phi \cdot g)(q_0,p_0,\lambda_0) = \mathcal{O}(h^\kappa)$, with $\kappa \geq 1$. Then the solution of eq. \eqref{eq:thm_existence_RK_subsystem}, $Q_i$, $p_i$, $P_i$ and $\Lambda_i$ satisfies:
\begin{equation*}
\begin{array}{rl}
Q_i &= q_0 + \dsum_{j = 1}^{\lambda} \dfrac{c_i^j h^j}{j!} DQ_{(j)}(q_0,p_0) + \mathcal{O}(h^{\lambda + 1})\\
p_i &= p_0 + \dsum_{j = 1}^{\lambda} \dfrac{c_i^j h^j}{j!} DP_{(j)}(q_0,p_0,\lambda_0) + \mathcal{O}(h^{\lambda + 1})\\
P_i &= p_0 + \dsum_{j = 1}^{\mu} \dfrac{c_i^j h^j}{j!} DP_{(j)}(q_0,p_0,\lambda_0) + \mathcal{O}(h^{\mu + 1})\\
\Lambda &= \lambda_0(q_0,p_0) + \dsum_{j = 1}^{\nu} \dfrac{c_i^j h^j}{j!} D\Lambda_{(j)}(q_0,p_0,\lambda_0) + \mathcal{O}(h^{\nu + 1})
\end{array}
\end{equation*}
where $\lambda_0(q_0,p_0)$ is implicitly defined by the condition $(D_1 \phi \cdot f)(q_0,p_0) + (D_2 \phi \cdot g)(q_0,p_0,\lambda_0(q_0,p_0)) = 0$, $\lambda = \min(\kappa + 1, q, \max(\hat{q} + 1, Q + 1))$, $\mu = \min(\kappa, \hat{q})$, $\nu = \min(\kappa - 1, q - 1)$, and $DQ_{(i)}$, $DP_{(i)}$ and $D\Lambda_{(i)}$ are functions composed by the derivatives of $f$, $g$ and $\phi$ evaluated at $(q_0, p_0, \lambda_0(q_0,p_0))$.
\end{lemma}

\begin{proof}
Following \cite{Jay93}, Lemma 4.3, we can use the implicit function theorem to obtain $\lambda_0(q_0,p_0) - \lambda_0 = \mathcal{O}(h^\kappa)$. Assume $(q(t), p(t), \lambda(t))$ is the exact solution of eq. \eqref{eq:partitioned_system} with $q(t_0) = q_0$, $p(t_0) = p_0$ and $\lambda(t_0) = \lambda_0$, and let $Q_i = q(t_0 + c_i h)$, $p_i = P_i = p(t_0 + c_i h)$ and $\Lambda_i = \lambda(t_0 + c_i h)$ in the result of theorem \ref{eq:thm_perturbed_distance_to_initial_values}. Finally set $\hat{Q}_i$, $\hat{p}_i$, $\hat{P}_i$ and $\hat{\Lambda}_i$ be the solution of eq. \eqref{eq:thm_perturbed_system} with $\hat{q}_0 = q_0$, $\hat{p}_0 = p_0$, $\hat{\lambda}_0 = \lambda_0(q_0,p_0)$ and $\theta = 0$. As we satisfy the conditions of theorem \ref{thm:impact_of_perturbations} we are left with:
\begin{align*}
\left\Vert \Delta Q_i \right\Vert &\leq C \left( h^{\kappa + 2} + h \left\Vert \delta_Q \right\Vert + h^2 \left\Vert \delta_p \right\Vert + h^2 \left\Vert \delta_P \right\Vert\right)\\
\left\Vert \Delta p_i \right\Vert &\leq C \left( h^{\kappa + 2} + h^2 \left\Vert \delta_Q \right\Vert + h \left\Vert \delta_p \right\Vert + h^2 \left\Vert \delta_P \right\Vert \right)\\
\left\Vert \Delta P_i \right\Vert &\leq C \left( h^{\kappa + 1} + h^2 \left\Vert \delta_Q \right\Vert + h \left\Vert \delta_p \right\Vert + h \left\Vert \delta_P \right\Vert\right)\\
\left\Vert \Delta \Lambda_i \right\Vert &\leq C \left( h^{\kappa} + \left\Vert \delta_Q \right\Vert + \left\Vert \delta_p \right\Vert\right)
\end{align*}
where we have made use of the fact that $\left\Vert \Delta \lambda_0 \right\Vert = \mathcal{O}(h^\kappa)$. What remains is to compute $\delta_Q$, $\delta_p$, $\delta_P$ to obtain the result we are after.

Inserting the exact solution into eq. \eqref{eq:thm_perturbed_system} we obtain:
\begin{align*}
q(t_0 + c_i h) &= q_0 + h \sum_{j = 1}^s a_{i j} f(q(t_0 + c_j h),p(t_0 + c_j h)) + h\delta_{Q,i}\\
&= q_0 + h \sum_{j = 1}^s a_{i j} \dot{q}(t_0 + c_i h) + h\delta_{Q,i}\nonumber\\
p(t_0 + c_i h) &= p_0 + h \sum_{j = 1}^s a_{i j} g(q(t_0 + c_j h),p(t_0 + c_j h),\lambda(t_0 + c_j h)) + h\delta_{p,i}\nonumber\\
&= p_0 + h \sum_{j = 1}^s a_{i j} \dot{p}(t_0 + c_i h) + h\delta_{p,i}\nonumber\\
p(t_0 + c_i h) &= p_0 + h \sum_{j = 1}^s \hat{a}_{i j} g(q(t_0 + c_j h),p(t_0 + c_j h),\lambda(t_0 + c_j h)) + h\delta_{P,i}\nonumber\\
&= p_0 + h \sum_{j = 1}^s \hat{a}_{i j} \dot{p}(t_0 + c_i h) + h\delta_{P,i}\nonumber
\end{align*}
\begin{equation*}
q(t_0 + c_i h) = q(t_0) + h \sum_{j=1}^{s} a_{i j} f(y(t_0 + c_j h), z(t_0 + c_j h)) + h \delta_i
\end{equation*}

Now, expanding in Taylor series about $t_0$ and taking into account that:
\begin{equation*}
y(x_0 + c_i h) = y(x_0) + \sum_{j = 1}^m \frac{1}{j!} y^{(j)}(x_0) c_i^j h^j + \mathcal{O}(h^{m+1})
\end{equation*}
we get:
\begin{align*}
\delta_{Q,i} = \frac{h^{q} q^{(q+1)}(x_0)}{q!} \left(\frac{c_i^{q+1}}{q+1} - \sum_{j=1}^{s} a_{i j} c_j^q\right) + \mathcal{O}(h^{q+1})\\
\delta_{p,i} = \frac{h^{q} p^{(q+1)}(x_0)}{q!} \left(\frac{c_i^{q+1}}{q+1} - \sum_{j=1}^{s} a_{i j} c_j^q\right) + \mathcal{O}(h^{q+1})\nonumber\\
\delta_{P,i} = \frac{h^{\hat{q}} p^{(\hat{q}+1)}(x_0)}{\hat{q}!} \left(\frac{c_i^{\hat{q}+1}}{\hat{q}+1} - \sum_{j=1}^{s} \hat{a}_{i j} c_j^{\hat{q}}\right) + \mathcal{O}(h^{\hat{q}+1})\nonumber
\end{align*}

Finally, we should be careful to note that according to eq. \eqref{eq:proof_perturbation_y} $\delta_{P,i}$ enters in $\Delta Q_i$ and $\Delta p_i$ multiplied by $A$ so we may invoke $C\hat{C}(Q)$. Thus, we have:
\begin{align*}
\left\Vert \Delta Q_i \right\Vert &\leq C \left( h^{\min(\kappa + 2, q+1, \max(\hat{q} + 2, Q + 2))}\right)\\
\left\Vert \Delta p_i \right\Vert &\leq C \left( h^{\min(\kappa + 2, q+1, \max(\hat{q} + 2, Q + 2))}\right)\\
\left\Vert \Delta P_i \right\Vert &\leq C \left( h^{\min(\kappa + 1, q+1, \hat{q} + 1)}\right)\\
\left\Vert \Delta \Lambda_i \right\Vert &\leq C \left( h^{\min(\kappa, q)}\right)
\end{align*}
which proves our lemma.
\end{proof}

\begin{remark}
For the Lobatto IIIA-B methods we have that $\hat{q} + 2 = Q = q = s$ and this result implies that:
\begin{equation*}
\begin{array}{rlcrl}
\left\Vert \Delta Q_i \right\Vert &= \mathcal{O}(h^{\min(\kappa + 2, s+1)}), &&\left\Vert \Delta P_i \right\Vert &= \mathcal{O}(h^{\min(\kappa + 1, s-1)}),\\
\left\Vert \Delta p_i \right\Vert &= \mathcal{O}(h^{\min(\kappa + 2, s+1)}), &&\left\Vert \Delta \Lambda_i \right\Vert &= \mathcal{O}(h^{\min(\kappa, s)}).
\end{array}
\end{equation*}
\end{remark}


For the development of the main theorem, on which the results of error and convergence rest, we will need the following definitions. 

\subsection{$R$-strings}
\noindent\textbf{$R$-string.} An $R$-string $\gamma$ of dimension $\dim \gamma = s$ is an ordered list of $s$ numbers $(\gamma_{(1)},...,\gamma_{(s)})$, where $\gamma_{(i)} \in \mathbb{N}_0$.

\noindent\textbf{Irreducible $R$-string.} An irreducible $R$-string $\gamma$ of $\dim \gamma = s$, is such that for $1 < i < s$ even, $\gamma_{(i)}$ and $\gamma_{(i+1)}$ are not simultaneously zero.

For our purposes an $R$-string $\gamma$ can be used as multi-index provided it is irreducible. They will appear in the terms:
\begin{equation*}
R_{\gamma} = \tilde{C}^{\gamma_{(1)}} \tilde{A}^{-1} \left[ \prod_{i = 2}^{\dim \gamma - 2}\tilde{C}^{\gamma_{(i)}} \tilde{A} \tilde{C}^{\gamma_{(i+1)}} \tilde{A}^{-1} \right] \tilde{C}^{\gamma_{(\dim \gamma)}}
\end{equation*}
of a certain Taylor expansion which play a crucial role in the next theorem that we prove.

\subsubsection{$R$-string operations}
\noindent\textbf{Left appending.} Given an irreducible $R$-string $\gamma$ of $\dim \gamma = s$ such that $\gamma_{(1)} \neq 0$, left appending gives a new $R$-string $\gamma' = (0,0,\gamma_{(1)},...,\gamma_{s)})$ of $\dim \gamma' = s+2$.

\noindent\textbf{Right appending.} Given an irreducible $R$-string $\gamma$ of $\dim \gamma = s$ such that $\gamma_{(s)} \neq 0$, right appending gives a new $R$-string $\gamma' = (\gamma_{(1)},...,\gamma_{(s)},0,0)$ of $\dim \gamma' = s+2$.

\noindent\textbf{Insertion.} For $1 \leq i < s$ odd, given an irreducible $R$-string $\gamma$ such that $\gamma_{(i)} \neq 0 \neq \gamma_{(i+1)}$, insertion gives a new $R$-string $$\gamma' = (\gamma_{(1)},...,\gamma_{(i)},0,0,\gamma_{(i+1)},...,\gamma_{(s)})$$ of $\dim \gamma' = s+2$.

\noindent\textbf{Left splitting.} For $1 \leq i \leq s$, given an irreducible $R$-string $\gamma$ such that $\gamma_{(i)} < 1$, left splitting gives a new $R$-string $\gamma' = (\gamma_{(1)},...,\gamma_{(i-1)},1,0,\gamma_{(i)}-1,...,\gamma_{(s)})$ of $\dim \gamma' = s+2$.

\noindent\textbf{Right splitting.} For $1 \leq i \leq s$, given an irreducible $s$-string $\gamma$ such that $\gamma_{(i)} < 1$, right splitting gives a new $R$-string $\gamma' = (\gamma_{(1)},...,\gamma_{(i)}-1,0,1,\gamma_{(i+1)},$ $...,\gamma_{(s)})$ of $\dim \gamma' = s+2$.

\noindent\textbf{Capping.} Given an irreducible $s$-string $\gamma$ such that $\gamma_{(1)} \neq 0 \neq \gamma_{(s)}$, capping gives a new $R$-string $\gamma' = (0,\gamma_{(1)},...,\gamma_{(s)},0)$ of $\dim \gamma' = s+2$.

\noindent\textbf{Left Diffusion.} For $1 < i \leq s$, given an irreducible $s$-string $\gamma$ such that $\gamma_{(i)} < 1$, diffusion gives a new $R$-string $\gamma' = (\gamma_{(1)},...,\gamma_{(i-1)}+1,\gamma_{(i)}-1,...,\gamma_{(s)})$ of $\dim \gamma' = s$.

\noindent\textbf{Right Diffusion.} For $1 \leq i < s$, given an irreducible $s$-string $\gamma$ such that $\gamma_{(i)} < 1$, diffusion gives a new $R$-string $\gamma' = (\gamma_{(1)},...,\gamma_{(i)}-1,\gamma_{(i+1)}+1,...,\gamma_{(s)})$ of $\dim \gamma' = s$.

Clearly all of these operations preserve irreducibility. Note that insertion can be absorbed into the splitting operations if we let $\gamma_{(i)} \leq 1$, but then we would need to add provisions so that the extended splitting operations preserve irreducibility.

\subsubsection{$R$-string classes and elementary $R$-strings}
We say that given two irreducible $R$-strings $\gamma$ and $\delta$ with $\left\vert \gamma \right\vert = \left\vert \delta \right\vert$ but $\dim \gamma$ and $\dim \delta$ not necessarily equal are in the same \emph{class} iff $R_{\gamma} = R_{\delta}$.

If they belong to the same class, then one can be derived from the other following certain rules. For a given order there are as many unique $R$ coefficients as \emph{elementary $R$-strings}. For each order, an elementary string is the shortest irreducible $R$-string (of even dimension) such that it cannot be derived from another elementary $R$-string via splitting, appending or insertion. The $n$-th order has $2^n$ elementary strings. The simplest elementary $R$-strings of a given order are of dimension $2$, i.e., $R$-strings $\gamma = (\gamma_{(1)},\gamma_{(2)})$ such that $\gamma_{(1)} + \gamma_{(2)} = n$, of which there are $n + 1$. The rest of the elementary strings can be obtained from these via diffusion and capping.

For $n = 3$ we know there are $2^3 = 8$ elementary $R$-strings. We have the following elementary $R$-strings of dimension $2$: $(0, 3)$, $(1,2)$, $(2,1)$, $(3,0)$. We may obtain the remaining four by capping and diffusion. First we may cap $(1,2)$ and $(2,1)$ to obtain $(0,1,2,0)$ and $(0,2,1,0)$ respectively. From these new elementary $R$-strings of dimension $4$ we obtain $(0,1,1,1)$ and $(1,1,1,0)$.

For $n = 4$ we know there are $2^4 = 16$ elementary strings. We have the following elementary $R$-strings of dimension 2: $(0,4)$, $(1,3)$, $(2,2)$, $(3,1)$, $(4,0)$. First we may cap $(1,3)$, $(2,2)$ and $(3,1)$ to obtain $(0,1,3,0)$, $(0,2,2,0)$ and $(0,3,1,0)$ respectively. From these new elementary $R$-strings we obtain $(0,1,3,0)$, $(0,1,2,1)$, $(0,1,1,2)$, $(0,2,2,0)$, $(1,1,1,1)$, $(1,2,1,0)$, $(2,1,1,0)$. Finally we may cap again the only $R$-string of dimension $4$ that admits capping, $(1,1,1,1)$, obtaining $(0,1,1,1,1,0)$.

As an example of derivation of strings of a class, let us take $(3,0)$. Applying the left appending operation we can obtain $(0,0,3,0)$. Applying the splitting operation to $(3,0)$ we obtain $(2,0,1,0)$ and $(1,0,2,0)$. The rest of the derived strings of the class can be obtained via further appending and/or splitting: $(0,0,2,0,1,0)$, $(0,0,1,0,2,0)$, $(1,0,1,0,1,0)$ and $(0,0,1,0,1,0,1,0)$.
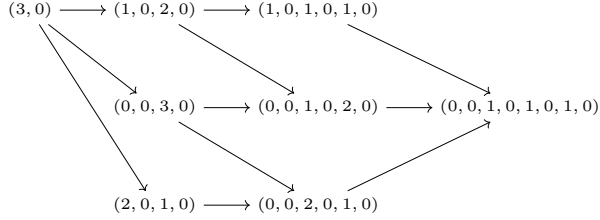
\begin{tiny}
	\begin{center}
		\begin{tikzcd}[column sep=normal, row sep=huge]
				(3, 0) \arrow[r] \arrow[dr] \arrow[ddr] & (1, 0, 2, 0) \arrow[r] \arrow[dr] & (1, 0, 1, 0, 1, 0) \arrow[dr]&\\
				                         & (0, 0, 3, 0) \arrow[r] \arrow[dr] & (0, 0, 1, 0, 2, 0) \arrow[r] & (0, 0, 1, 0, 1, 0, 1, 0)\\
				                         & (2, 0, 1, 0) \arrow[r] & (0, 0, 2, 0, 1, 0) \arrow[ur] &
		\end{tikzcd}
	\end{center}
\end{tiny}
Another example where we may use insertion is $(0,1,1,1)$, which yields $(0,1,0,0,1,1)$. The rest of the elements of the class are obtained via appending $(0,1,1,1,0,0)$ and $(0,1,0,0,1,1,0,0)$.
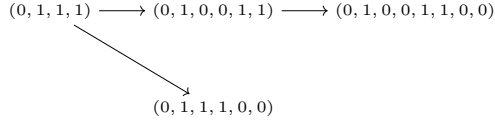
\begin{tiny}
	\begin{center}
		\begin{tikzcd}[column sep=normal, row sep=huge]
				(0,1,1,1) \arrow[r] \arrow[dr] & (0,1,0,0,1,1) \arrow[r] & (0,1,0,0,1,1,0,0)\\
				                               & (0,1,1,1,0,0)           & 
		\end{tikzcd}
	\end{center}
\end{tiny}

\begin{theorem}
\label{thm:RK_total_error}
In addition to the hypotheses of theorem \ref{thm:impact_of_perturbations}, suppose that $A$ and $\hat{A}$ are simplectic conjugated and, $C(q)$, $\hat{C}(r)$, $D(r)$, $\hat{D}(q)$, $D\hat{D}(p-r)$, $\hat{D}D(p-q)$ and \ref{itm:H3} hold. Furthermore, $(D_1 \phi \cdot f)(q_0,p_0) + (D_2 \phi \cdot g)(q_0,p_0,\lambda_0) = \mathcal{O}(h^\kappa)$, with $\kappa \geq 1$. Then we have:
\begin{align}
\left\Vert \Delta Q_s \right\Vert &= \Delta q_0 \label{eq:RK_total_error_Q}\\
&+ \mathcal{O}\left( h \left\Vert \Delta p_0 \right\Vert + h^{m + 2} \left\Vert \Delta \lambda_0 \right\Vert + h \left\Vert \delta_Q \right\Vert + h^2 \left\Vert \delta_p \right\Vert + h^2 \left\Vert \delta_P \right\Vert + h \left\Vert \theta \right\Vert \right)\nonumber\\
\left\Vert \Delta p_s \right\Vert &= \Pi_{1,0}(q_0, p_0, \lambda_0) \Delta q_0 + \Pi_{2,0}(q_0, p_0, \lambda_0) \Delta p_0\label{eq:RK_total_error_P}\\
&+ \mathcal{O}\left( h^{m + 2} \left\Vert \Delta \lambda_0 \right\Vert + h^2 \left\Vert \delta_Q \right\Vert + h \left\Vert \delta_p \right\Vert + h^2 \left\Vert \delta_P \right\Vert + \left\Vert \theta \right\Vert \right)\nonumber\\
\left\Vert \Delta \Lambda_s \right\Vert &= \mathcal{R}_A(\infty) \Delta \lambda_0\label{eq:RK_total_error_Lam}\\
&+ \mathcal{O}\left( \left\Vert \Delta q_0 \right\Vert + \left\Vert \Delta p_0 \right\Vert + h \left\Vert \Delta \lambda_0 \right\Vert + \left\Vert \delta_Q \right\Vert + \left\Vert \delta_p \right\Vert + \left\Vert \theta \right\Vert/h \right)\nonumber
\end{align}
where $m = \min(\kappa - 1, q - 1, r, p - q, p - r)$, $\mathcal{R}_A$ is the stability function of the method $A$, $\Pi_{1,0} = - D_3 g (D_2 \phi D_3 g) D_1 \phi$ and $\Pi_{2,0} = I_n - D_3 g (D_2 \phi D_3 g) D_2 \phi$.
\end{theorem}

\begin{proof}
This proof follows closely that of \cite{Jay93}, theorem 4.4. The idea is to take the results from theorem \ref{thm:impact_of_perturbations} and perform a Taylor expansion of each term, focusing on the $s$-th component. Just as in \cite{Jay93}, the important result here is the $h^{m+2}$ factor in front of $\left\Vert\Delta\lambda_0\right\Vert$, which means that we need to pay special attention to $\Delta \Lambda_1$.

In our case $\Delta \Lambda_s$ coincides with Jay's $\Delta Z_s$ without changes. The differences appear in the rest of the components, where having two sets of Runge-Kutta coefficients makes the Taylor expansion of the terms and the tracking of each component much more difficult. We want $\Delta Q_s$ and $\Delta p_s$, as $\Delta P_s$ is not an external stage / nodal value. Thus, we will need to expand eq. \eqref{eq:proof_perturbation_y}. Unfortunately this depends on eq. \eqref{eq:proof_perturbation_Y}. Let us first solve this latter equation for $\Delta Y$:
\begin{align*}
\Delta Y &= \left(I - h \left(\leftidx{^A}{\phantom{.}}{_{\hat{A}}}\right) P_{A} D_y F\right)^{-1}\\
&\times \left[ \Pi_{\hat{A}}^{A} \left( \Delta \eta + h \delta_y \right) + h \left(\leftidx{^A}{\phantom{.}}{_{\hat{A}}}\right) P_{A} D_z F_1 \Delta \Lambda_1 \vphantom{\left(\leftidx{^0}{\phantom{.}}{_{\left(D_2 \phi A D_3 g\right)^{-}}}\right)}\right.\nonumber\\
&- \left.\left(\leftidx{^A}{\phantom{.}}{_{\hat{A}}}\right) D_z \tilde{F} \left(\leftidx{^0}{\phantom{.}}{_{\left(D_2 \phi A D_3 g\right)^{-}}}\right)\theta + h \left(\delta_Y - \delta_y\right)\right]\nonumber
\end{align*}

We then need to insert this in eq. \eqref{eq:proof_perturbation_y}. From here on we will forget about all terms except for the ones with $\Delta \Lambda_1$, as the rest vary little from what was found in Theorem \ref{thm:impact_of_perturbations} and they can be easily obtained, thus barring the need to carry them around any longer.
\begin{align*}
\Delta y &= h \Pi_A \left(\leftidx{^A}{\phantom{.}}{_{A}}\right) D_z F_1 \Delta \Lambda_1\\
&+ h^2 \Pi_A \left(\leftidx{^A}{\phantom{.}}{_{A}}\right) D_y F \left(I - h \left(\leftidx{^A}{\phantom{.}}{_{\hat{A}}}\right) P_{A} D_y F\right)^{-1} \left(\leftidx{^A}{\phantom{.}}{_{\hat{A}}}\right) P_{A} D_z F_1 \Delta \Lambda_1\nonumber\\
&+ ...\nonumber
\end{align*}

The first term can be expanded just as in \cite{Jay93}, as there is no $\hat{A}$ involved, giving us $\mathcal{O}(h^{m+2} \left\Vert\Delta\lambda_0\right\Vert)$ as expected. The second term is where the real changes appear. Let us begin with the right-most part of the term, $\left(\leftidx{^A}{\phantom{.}}{_{\hat{A}}}\right) P_{A} D_z F_1 \Delta \Lambda_1$. We have that:
\begin{align*}
P_{A} &= \left[\begin{smallmatrix}
1 \otimes I_{n} &                     0 &               0 &                     0\\
              0 & I_{s-1} \otimes I_{n} &               0 &                     0\\
              0 &                     0 & 1 \otimes I_{n} &                     0\\
-\tilde{X}_1 \left(\tilde{A}_1 \otimes I_{n}\right) & -\tilde{X}_1\left(\tilde{A} \otimes I_{n}\right) & -\tilde{X}_2 \left(\tilde{A}_1 \otimes I_{n}\right) & I_{s-1} \otimes I_{n} - \tilde{X}_2\left(\tilde{A} \otimes I_{n}\right)
\end{smallmatrix}
\right]
\end{align*}
where $\tilde{X}_i$ was already defined in eq. \eqref{eq:def_X_i} and where we have used the fact that $A_1^1$ is zero and $A^1$ is a zero vector.

For the product $P_{A} D_z F_1 \Delta \Lambda_1$ we only need to worry about the component $I_{s-1} \otimes I_{n} - \tilde{X}_2\left(\tilde{A} \otimes I_{n}\right)$, as the rest do not connect with $\Delta \Lambda_1$.
\begin{align}
&\left(\leftidx{^A}{\phantom{.}}{_{\hat{A}}}\right) P_{A} D_z F_1 \Delta \Lambda_1 = \left[\begin{smallmatrix}
A_      1^1 \otimes I_{n} &       A^1 \otimes I_{n} &                                 0 &                             0\\
\tilde{A}_1 \otimes I_{n} & \tilde{A} \otimes I_{n} &                                 0 &                             0\\
                        0 &                       0 &         \hat{A}_1^1 \otimes I_{n} &       \hat{A}^1 \otimes I_{n}\\
                        0 &                       0 &   \hat{\tilde{A}}_1 \otimes I_{n} & \hat{\tilde{A}} \otimes I_{n}
\end{smallmatrix}
\right]\label{eq:proof_RK_total_APADFDL}\\
&\times \left[\begin{smallmatrix}
1 \otimes I_{n} &                     0 &               0 &                     0\\
              0 & I_{s-1} \otimes I_{n} &               0 &                     0\\
              0 &                     0 & 1 \otimes I_{n} &                     0\\
-\tilde{X}_1 \left(\tilde{A}_1 \otimes I_{n}\right) & -\tilde{X}_1\left(\tilde{A} \otimes I_{n}\right) & -\tilde{X}_2 \left(\tilde{A}_1 \otimes I_{n}\right) & I_{s-1} \otimes I_{n} - \tilde{X}_2\left(\tilde{A} \otimes I_{n}\right)
\end{smallmatrix}
\right] \left[\begin{smallmatrix}
                               0\\
                               0\\
D_3 \tilde{g}_0 \Delta \Lambda_1\\
                               0
\end{smallmatrix}
\right]\nonumber
\end{align}

Inside $\tilde{X}_2$ we find the product $D_2 \tilde{\phi} \left(\tilde{A} \otimes I_{n}\right) D_3 \tilde{g} \left(\tilde{A} \otimes I_{m}\right)^{-1}$ composed of:
\begin{equation*}
D_2 \tilde{\phi} = \sum_{i = 0}^\omega h^i \tilde{C}^i \otimes D_2 \tilde{\phi}_{i} + \mathcal{O}(h^{\omega + 1})
\end{equation*}
\begin{equation*}
\left(\tilde{A} \otimes I_{n}\right) D_3 \tilde{g} \left(\tilde{A} \otimes I_{m}\right)^{-1} = \sum_{i = 0}^\omega h^i \tilde{A} \tilde{C}^i \tilde{A}^{-1} \otimes D_3 \tilde{g}_{i} + \mathcal{O}(h^{\omega + 1})
\end{equation*}
which results in:
\begin{equation*}
D_2 \tilde{\phi} \left(\tilde{A} \otimes I_{n}\right) D_3 \tilde{g} \left(\tilde{A} \otimes I_{m}\right)^{-1} = \sum_{0 \leq i + j \leq \omega}^\omega h^{i + j} \tilde{C}^i \tilde{A} \tilde{C}^j \tilde{A}^{-1} \otimes D_2 \tilde{\phi}_i D_3 \tilde{g}_j + \mathcal{O}(h^{\omega + 1})
\end{equation*}

Inversion of this product can be carried out as a Taylor expansion resulting in a so-called von Neumann series $(I - T)^{-1} = \sum_{i=0}^{\infty} T^i$. Let us rewrite the former expression:
\begin{align*}
&D_2 \tilde{\phi} \left(\tilde{A} \otimes I_{n}\right) D_3 \tilde{g} \left(\tilde{A} \otimes I_{m}\right)^{-1}\\
&= \left( I_{s-1} \otimes I_n + \sum_{1 < i + j \leq \omega}^\omega h^{i + j} \tilde{C}^i \tilde{A} \tilde{C}^j \tilde{A}^{-1} \otimes D_2 \tilde{\phi}_i D_3 \tilde{g}_j \left(D_2 \tilde{\phi}_0 D_3 \tilde{g}_0\right)^{-1}\right)\\
&\times \left( I_{s-1} \otimes D_2 \tilde{\phi}_0 D_3 \tilde{g}_0\right) + \mathcal{O}(h^{\omega + 1})\\
&= \left( I_{s-1} \otimes I_n + \sum_{1 < i + j \leq \omega}^\omega h^{i + j} \tilde{C}^i \tilde{A} \tilde{C}^j \tilde{A}^{-1} \otimes D_2 \tilde{\phi}_i D_3 \tilde{g}_j \vartriangle\right)\\
&\times \left( I_{s-1} \otimes \triangledown\right) + \mathcal{O}(h^{\omega + 1})\\
&= \left( I_{s-1} \otimes I_n - \sum_{1 < \left\vert\alpha\right\vert}^\omega - h^{\left\vert \alpha\right\vert} N_{\alpha} \otimes M_{\alpha} \right) \times \left( I_{s-1} \otimes \triangledown\right) + \mathcal{O}(h^{\omega + 1})
\end{align*}
with $\alpha$ multi-index of $\dim\alpha = 2$. For instance, for $\left\vert\alpha\right\vert = 3$ we have $\alpha_1 = (3,0), \alpha_2 = (2,1), \alpha_3 = (1,2), \alpha_1 = (0,3)$, and the corresponding terms $N_{\alpha} \otimes M_{\alpha}$ are:
\begin{align*}
N_{(3,0)} \otimes M_{(3,0)} &= \tilde{C}^3 \otimes D_2 \tilde{\phi}_3 D_3 \tilde{g}_0 \vartriangle\\
N_{(2,1)} \otimes M_{(2,1)} &= \tilde{C}^2 \tilde{A} \tilde{C} \tilde{A}^{-1} \otimes D_2 \tilde{\phi}_2 D_3 \tilde{g}_1 \vartriangle\\
N_{(1,2)} \otimes M_{(1,2)} &= \tilde{C} \tilde{A} \tilde{C}^2 \tilde{A}^{-1} \otimes D_2 \tilde{\phi}_1 D_3 \tilde{g}_2 \vartriangle\\
N_{(0,3)} \otimes M_{(0,3)} &= \tilde{A} \tilde{C}^3 \tilde{A}^{-1} \otimes D_2 \tilde{\phi}_0 D_3 \tilde{g}_3 \vartriangle
\end{align*}

We have also made use of the short-hand notation $\triangledown = D_2 \tilde{\phi}_0 D_3 \tilde{g}_0$ and $\vartriangle = \left(D_2 \tilde{\phi}_0 D_3 \tilde{g}_0\right)^{-1}$.

Paying attention to the non-commutativity of the series we obtain:
\begin{align*}
&\left( D_2 \tilde{\phi} \left(\tilde{A} \otimes I_{n}\right) D_3 \tilde{g} \left(\tilde{A} \otimes I_{m}\right)^{-1}\right)^{-1}\\
&= \left( I_{s-1} \otimes \vartriangle\right) \times \left( \sum_{\left\vert \beta\right\vert = 0}^\omega (-1)^{\frac{\dim \beta}{2}} h^{\left\vert \beta\right\vert} N_{\beta} \otimes M_{\beta} \right) + \mathcal{O}(h^{\omega + 1})
\end{align*}
with $\beta$ multi-index of $\dim\beta \leq 2 \omega$, even, and such that for $i$ \textbf{odd} $\beta_{(i)}$ and $\beta_{(i+1)}$ are never both 0. For instance, for $\left\vert\beta\right\vert = 2$ we have, for $\dim \beta = 2$, $\beta_{1,1} = (2,0), \beta_{1,2} = (1,1), \beta_{1,3} = (0,2)$, and for $\dim \beta = 4$ we have $\beta_{2,1} = (1,0,1,0), \beta_{2,2} = (1,0,0,1), \beta_{2,3} = (0,1,1,0), \beta_{2,4} = (0,1,0,1)$. $(0,0,1,1)$ and $(1,1,0,0)$ are not allowed as they contain two contiguous zeros in odd and even position. Some examples of the corresponding terms $N_{\beta} \otimes M_{\beta}$ are:
\begin{align*}
N_{(1,1)} \otimes M_{(1,1)} &= \tilde{C} \tilde{A} \tilde{C} \tilde{A}^{-1} \otimes D_2 \tilde{\phi}_1 D_3 \tilde{g}_1 \vartriangle\\
N_{(0,1,1,0)} \otimes M_{(0,1,1,0)} &= \tilde{A} \tilde{C} \tilde{A}^{-1} \tilde{C} \otimes D_2 \tilde{\phi}_0 D_3 \tilde{g}_1 \vartriangle D_2 \tilde{\phi}_1 D_3 \tilde{g}_0 \vartriangle
\end{align*}

We need to include the restriction on elements such as $(0,0,1,1)$ as a double-counting prevention of sorts. We can understand this by checking what its associated $M_{(0,0,1,1)}$ would look like: $$D_2 \tilde{\phi}_0 D_3 \tilde{g}_0 \vartriangle D_2 \tilde{\phi}_1 D_3 \tilde{g}_1 \vartriangle = \triangledown \vartriangle D_2 \tilde{\phi}_1 D_3 \tilde{g}_1 \vartriangle = D_2 \tilde{\phi}_1 D_3 \tilde{g}_1 \vartriangle = M_{(1,1)}\; .$$

Moving on to the next computation, we sandwich the expression between $D_3 \tilde{g} \left(\tilde{A} \otimes I_{m}\right)$ and $D_2 \tilde{\phi}$ to obtain:
\begin{align*}
\tilde{X}_2 &= D_3 \tilde{g} \left(\tilde{A} \otimes I_{m}\right) \left( D_2 \tilde{\phi} \left(\tilde{A} \otimes I_{n}\right) D_3 \tilde{g} \left(\tilde{A} \otimes I_{m}\right)^{-1}\right)^{-1} D_2 \tilde{\phi}\\
&= \left( \sum_{\left\vert \gamma\right\vert = 0}^\omega (-1)^{\frac{\dim \gamma}{2} - 1} h^{\left\vert \gamma\right\vert} R_{\gamma} \otimes S_{\gamma} \right) + \mathcal{O}(h^{\omega + 1})
\end{align*}
where:
\begin{align*}
R_{\gamma} &= \tilde{C}^{\gamma_{(1)}} \tilde{A}^{-1} \left[ \prod_{i = 2}^{\dim \gamma - 2}\tilde{C}^{\gamma_{(i)}} \tilde{A} \tilde{C}^{\gamma_{(i+1)}} \tilde{A}^{-1} \right] \tilde{C}^{\gamma_{(\dim \gamma)}}\\
S_{\gamma} &= D_3 \tilde{g}_{\gamma_{(1)}} \vartriangle \left[ \prod_{i = 2}^{\dim \gamma - 2} D_2 \tilde{\phi}_{\gamma_{(i)}} D_3 \tilde{g}_{\gamma_{(i+1)}} \vartriangle\right] D_2 \tilde{\phi}_{\gamma_{(\dim \gamma)}}
\end{align*}
with $\gamma$ multi-index of $\dim\gamma \leq 2 \omega$, even, and such that for $i$ \textbf{even} $\gamma_{(i)}$ and $\gamma_{(i+1)}$ are never both 0, i.e., $\gamma$ is an irreducible $R$-string.

This structure looks quite complicated as it is, and it does not seem to lend itself to easy groupings of symbol combinations $R_{\gamma}$. Nevertheless, it can be done with the help of the $R$-string classes we introduced before.

Once we have derived $\tilde{X}_2$ and essentially $P_A$, we can finally tackle the full product $\left(\leftidx{^A}{\phantom{.}}{_{\hat{A}}}\right) P_{A} D_z F_1 \Delta \Lambda_1$. If we perform the matrix multiplications in eq. \eqref{eq:proof_RK_total_APADFDL} we get:
\begin{equation*}
\left(\leftidx{^A}{\phantom{.}}{_{\hat{A}}}\right) P_{A} D_z F_1 \Delta \Lambda_1 = \left[\begin{smallmatrix}
0\\
0\\
\hat{A}^1_1 \otimes D_3 \tilde{g}_0 \Delta \Lambda_1\\
\hat{\tilde{A}}_1 \otimes D_3 \tilde{g}_0 \Delta \Lambda_1 - \left(\hat{\tilde{A}} \otimes I_{n}\right) \tilde{X}_2 \left(\tilde{A}_1 \otimes D_3 \tilde{g}_0 \Delta \Lambda_1\right)\\
\end{smallmatrix}
\right]
\end{equation*}

What is important here is that we are multiplying by $D_3 \tilde{g}_0$ on the right. In terms of strings this means appending one zero to the right, which makes a big part of the expansion vanish, as we will see in proposition \ref{prop:R_string_elimination}. This result gives us valuable information about the series expansion:
\begin{equation*}
\left(\Pi_{A, 0} + \mathcal{O}(h)\right) \left(\leftidx{^A}{\phantom{.}}{_{A}}\right) \left(D_y F_0 + \mathcal{O}(h) \right) \left(I - \mathcal{O}(h)\right)^{-1} \left(\leftidx{^A}{\phantom{.}}{_{\hat{A}}}\right) P_{A} D_z F_1 \Delta \Lambda_1
\end{equation*}

At order 0, for the last component we get that the combination
\begin{equation*}
e_{s}^T \left[\begin{array}{cc}
      A_1^1 &       A^1\\
\tilde{A}_1 & \tilde{A}
\end{array}
\right] \left(\left[\begin{array}{c}
        \hat{A}_1^1\\
  \hat{\tilde{A}}_1
\end{array}
\right] - \left[\begin{array}{c}
								 0\\
\hat{\tilde{A}} \tilde{A}^{-1} \tilde{A}_1
\end{array}
\right]\right) = 0
\end{equation*}
where the vector $e_{s}^T = (0,...,0,1)$, with $\dim e_{s} = s$. For a method satisfying \ref{itm:H3} we have that $e_{s}^T A = b$. Using the notation:
\begin{equation*}
A^{-} = \left[\begin{array}{cc}
      0_1^1 &            0^1\\
\tilde{0}_1 & \tilde{A}^{-1}
\end{array}
\right], \quad A_1 = \left[\begin{array}{c}
      A_1^1\\
\tilde{A}_1
\end{array}
\right], \quad \hat{A}_1 = \left[\begin{array}{c}
      \hat{A}_1^1\\
\hat{\tilde{A}}_1
\end{array}
\right]
\end{equation*}
we may write this expression in shorter form as $e_s^T A (\hat{A}_1 - \hat{A} A^{-} A_1)$.

At order $h$ we still do not have all the terms that arise from the expansion but we already have an interesting combination that must also vanish.
\begin{equation*}
e_{s}^T \left[\begin{array}{cc}
      A_1^1 &       A^1\\
\tilde{A}_1 & \tilde{A}
\end{array}
\right] \left[\begin{array}{cc}
      \hat{A}_1^1 &       \hat{A}^1\\
\hat{\tilde{A}}_1 & \hat{\tilde{A}}
\end{array}
\right]\left[\begin{array}{c}
				                  0\\
\tilde{C} \tilde{A}^{-1} \tilde{A}_1
\end{array}
\right] = 0
\end{equation*}
We can write this combination as $e_s^T A \hat{A} C A^{-} A_1$.

In fact the two vanishing combinations hint at the template for the rest of the vanishing combinations: $e_s^T ...\, A\, ...\, (\hat{A}_1 - \hat{A} A^{-} A_1)$ and $e_s^T ...\, A\, ...\, \hat{A}\, ...\, C A^{-} A_1$.

As we will see later, for all combinations there will always be at least one $\hat{A}$ (which the first template already includes) and one $A$, as can be readily seen below:
\begin{equation*}
\Pi_{A} \underline{\underline{\left(\leftidx{^A}{\phantom{.}}{_{A}}\right)}} D_y F \left(I - h \left(\leftidx{^A}{\phantom{.}}{_{\hat{A}}}\right) P_{A} D_y F\right)^{-1} \underline{\underline{\left(\leftidx{^A}{\phantom{.}}{_{\hat{A}}}\right)}} P_{A} D_z F_1 \Delta \Lambda_1
\end{equation*}
As we grow in order, up to order $n$, combinations of $A$, $\hat{A}$, $C$ and $A^{-}CA$ show up such that their number adds up to $n + 1$. These originate from $P_{A}$ itself, as well as $\Pi_{A}$, $D_y F$ and $\left(I - h \left(\leftidx{^A}{\phantom{.}}{_{\hat{A}}}\right) P_{A} D_y F\right)^{-1}$, as we will soon see.

With all the knowledge we got from our string analysis it can be shown that the expansion of $\left(\leftidx{^A}{\phantom{.}}{_{\hat{A}}}\right) P_{A} D_z F_1 \Delta \Lambda_1$ takes the form:
\begin{align*}
&(\hat{A}_1 - \hat{A} A^{-} A_1) \otimes D_3 g_0 \Delta \Lambda_1\\
&+ \sum_{\left\vert\rho\right\vert = 0}^{\omega} h^{\left\vert\rho\right\vert + 1} \left[\hat{A} \left(\prod_{i = 1}^{\dim \rho - 1} C^{\rho_i} A^{-} C^{\rho_{i+1}} A\right) C \otimes O_\rho \Delta \Lambda_1\right] + \mathcal{O}(h^{\omega + 1})\nonumber
\end{align*}
where $O_\rho$ is a term composed by multiplication of $\vartriangle$ and derivatives of $g$ and $\phi$ evaluated at the initial condition.

For the remaining expansions we do not need to be as precise as with this last one as there will not be cancellations due to signs. Thus we will only care about the different symbol combinations that arise.

The object $\left(I - h \left(\leftidx{^A}{\phantom{.}}{_{\hat{A}}}\right) P_{A} D_y F\right)^{-1}$ is the most involved of all of them as it is a matrix term that couples the $\Delta Q$ and $\Delta P$ equations. The matrix multiplied by $h$ is:
\begin{equation*}
\left(\leftidx{^A}{\phantom{.}}{_{\hat{A}}}\right) P_{A} D_y F = {\left[
\begin{smallmatrix}
A D_1 f & A D_2 f\\
-\hat{A} \left[D_1 f X_1 A + D_1 g (X_2 A - \mathds{1})\right] & -\hat{A} \left[D_2 f X_1 A + D_2 g (X_2 A - \mathds{1})\right]
\end{smallmatrix}
\right]}
\end{equation*}

If we write $I - h \left(\leftidx{^A}{\phantom{.}}{_{\hat{A}}}\right) P_{A} D_y F$ as:
\begin{equation*}
\left[
\begin{array}{cc}
1 - K &   - L\\
  - M & 1 - N
\end{array}
\right]
\end{equation*}
where the matrices $K, L, M, N$ are $\mathcal{O}(h)$, then its inverse must be:
\begin{equation*}
\left[
\begin{array}{cc}
W &   X\\
Y &   Z
\end{array}
\right]
\end{equation*}
with:
\begin{align*}
W &= (1 - K - L (1 - N)^{-1} M)^{-1},\\
X &= (1 - K)^{-1} L (1 - N - M (1 - K)^{-1} L)^{-1},\\
Y &= (1 - N)^{-1} M (1 - K - L (1 - N)^{-1} M)^{-1},\\
Z &= (1 - N - M (1 - K)^{-1} L)^{-1}.
\end{align*}

The only terms we are interested in are $X$ and $Z$, as those are the only ones that connect with $\Delta \Lambda_1$. The Taylor expansion of any of these terms is a daunting task given the amount of nested expansions of non-commutative terms involved. Instead we deem it sufficient to analyse the symbolic expansion found via CAS up to order 4 and draw our conclusions from there. In our case we will use the SymPy library for Python for the actual computations.

Before we begin analysing terms, it is interesting to check the form of $X$ and $Z$. We can see that $X = (1 - K)^{-1} L Z$. This means that once we know the behaviour of $Z$, that of $X$ will be easy to derive. Also from this we can easily see that all the resulting symbol combinations of $X$ must necessarily start with the coefficient matrix $A$, while for $Z$ they must start with the coefficient matrix $\hat{A}$ with the exception of the zero-th order term. In fact this is also true for $W$ and $Y$ respectively, being $W$ the one with non-zero zero-th order term.

The expansion of $Z$ (and $Y$) shows the following symbol combinations up to order 3:
\begin{footnotesize}
\begin{center}
    \begin{tabular}{ | c | c | c | c |}
    \hline
    Order &                Term &              1 Substitution &       2 Substitutions  \\ \hline
        1 &           $\hat{A}$ &                             & $\vphantom{\sum_{j}^{N^1}}$ \\ \hline
        2 &         $\hat{A} A$ &                             & $\vphantom{\sum_{j}^{N^1}}$ \\ \cline{2-4}
          &         $\hat{A} C$ &         $\hat{A} A^{-} C A$ & $\vphantom{\sum_{j}^{N^1}}$ \\ \cline{2-4}
          &         $\hat{A}^2$ &                             & $\vphantom{\sum_{j}^{N^1}}$ \\ \hline
        3 &       $\hat{A} A^2$ &                             & $\vphantom{\sum_{j}^{N^1}}$ \\ \cline{2-4}
          &       $\hat{A} A C$ &                             & $\vphantom{\sum_{j}^{N^1}}$ \\ \cline{2-4}
          & $\hat{A} A \hat{A}$ &                             & $\vphantom{\sum_{j}^{N^1}}$ \\ \cline{2-4}
          &       $\hat{A} C A$ &       $\hat{A} A^{-} C A^2$ & $\vphantom{\sum_{j}^{N^1}}$ \\ \cline{2-4}
          &       $\hat{A} C^2$ &       $\hat{A} A^{-} C A C$ & $\vphantom{\sum_{j}^{N^1}} \hat{A} A^{-} C^2 A$ \\ \cline{3-4}
          &                     &       $\hat{A} C A^{-} C A$ & $\vphantom{\sum_{j}^{N^1}}$ \\ \cline{2-4}
          & $\hat{A} C \hat{A}$ & $\hat{A} A^{-} C A \hat{A}$ & $\vphantom{\sum_{j}^{N^1}}$ \\ \cline{2-4}
          &       $\hat{A}^2 A$ &                             & $\vphantom{\sum_{j}^{N^1}}$ \\ \cline{2-4}
          &       $\hat{A}^2 C$ &       $\hat{A}^2 A^{-} C A$ & $\vphantom{\sum_{j}^{N^1}}$ \\ \cline{2-4}
          &         $\hat{A}^3$ &                             & $\vphantom{\sum_{j}^{N^1}}$ \\
    \hline
    \end{tabular}
\end{center}
\end{footnotesize}

As for the expansion of $X$ (and $W$), we get:
\begin{footnotesize}
\begin{center}
    \begin{tabular}{ | c | c | c | c |}
    \hline
    Order &          Term &        1 Substitution & 2 Substitutions  \\ \hline
        1 &           $A$ &                       & $\vphantom{\sum_{j}^{N^1}}$ \\ \hline
        2 &         $A^2$ &                       & $\vphantom{\sum_{j}^{N^1}}$ \\ \cline{2-4}
          &         $A C$ &                       & $\vphantom{\sum_{j}^{N^1}}$ \\ \cline{2-4}
          &   $A \hat{A}$ &                       & $\vphantom{\sum_{j}^{N^1}}$ \\ \hline
        3 &         $A^3$ &                       & $\vphantom{\sum_{j}^{N^1}}$ \\ \cline{2-4}
          &       $A^2 C$ &                       & $\vphantom{\sum_{j}^{N^1}}$ \\ \cline{2-4}
          & $A^2 \hat{A}$ &                       & $\vphantom{\sum_{j}^{N^1}}$ \\ \cline{2-4}
          &       $A C A$ &                       & $\vphantom{\sum_{j}^{N^1}}$ \\ \cline{2-4}
          &       $A C^2$ &                       & $\vphantom{\sum_{j}^{N^1}}$ \\ \cline{2-4}
          & $A C \hat{A}$ &                       & $\vphantom{\sum_{j}^{N^1}}$ \\ \cline{2-4}
          & $A \hat{A} A$ &                       & $\vphantom{\sum_{j}^{N^1}}$ \\ \cline{2-4}
          & $A \hat{A} C$ & $A \hat{A} A^{-} C A$ & $\vphantom{\sum_{j}^{N^1}}$ \\ \cline{2-4}
          & $A \hat{A}^2$ &                       & $\vphantom{\sum_{j}^{N^1}}$ \\
    \hline
    \end{tabular}
\end{center}
\end{footnotesize}

Focusing on $Z$, we can see that for order 2 we append either an $A$, $C$ or $\hat{A}$ to the right of the order 1 terms. Also note that $C$ can be substituted by the combination $A^{-} C A$ once, so long as the preceding symbol in the term without substitutions is not an $A$. We find the same relation between order 3 and order 2, and (although not shown here) for order 4 and order 3. Thus the pattern of construction of terms to arbitrary order seems clear for $Z$.

Focusing now on $X$, and taking into account the discussion at the beginning of the section, we can see that all the terms in $Z$ will show up multiplied by $(1 - K)^{-1} L$. The symbols this factor adds at order $n$ are $A \times [ (n-1)\text{-element variations  of} \left\lbrace A,C\right\rbrace]$. In practice, what we observe with the symbolic expansion is that every single term without substitution at order $n$ in $Z$ appears in $X$ with the first $\hat{A}$ exchanged by $A$. As for substituted terms, at order $n$ we find all substitution terms from $Z$ up to order $n-1$ with a corresponding pre-factor $\left\lbrace A,C\right\rbrace$. For instance, if we take $\hat{A} A^{-} C A$, which is of order 2 for $Z$, we will find it as $A \hat{A} A^{-} C A$ at order 3, and as $A C\hat{A} A^{-} C A$ and $A^2 \hat{A} A^{-} C A$ at order 4 and so on. An easier way to put this is that the construction of terms for $X$ is the same as for $Z$ with the restriction that substitutions $C \mapsto A^{-} C A$ can only appear after the first $\hat{A}$ that show up.

Let us finally expand the term $\Pi_{A} \left(\leftidx{^A}{\phantom{.}}{_{A}}\right) D_y F$. For the projector $\Pi_A$, (see eq.\eqref{eq:proof_RK_total_PiA}), its $\tilde{\Pi}_{i,A}$ are very similar to the terms $\tilde{X}_{i}$ that we have already studied:
\begin{equation*}
\tilde{\Pi}_{i,A} = \sum_{\left\vert \gamma\right\vert = 0}^\omega h^{\left\vert \gamma\right\vert} \left[ \prod_{j = 1}^{\dim \gamma - 1}\tilde{C}^{\gamma_{(j)}} \tilde{A} \tilde{C}^{\gamma_{(j+1)}} \tilde{A}^{-1} \right] \otimes \tilde{\Pi}_{i, A, \gamma} + \mathcal{O}(h^{\omega + 1})
\end{equation*}

It is important to note that as we have the product $\Pi_{A} \left(\leftidx{^A}{\phantom{.}}{_{A}}\right)$, we will always have one $A^{-}$ less than the number of $A$s, which prevents $A C^k A^{-}$ terms from appearing at the very end of a symbol combination.

For the Jacobian $D_y F$ we have:
\begin{align*}
D_y F = \left[\begin{smallmatrix}
D_1 f_1 &             0 & D_2 f_1 &             0\\
      0 & D_1 \tilde{f} &       0 & D_2 \tilde{f}\\
D_1 g_1 &             0 & D_2 g_1 &             0\\
      0 & D_1 \tilde{g} &       0 & D_2 \tilde{g}
\end{smallmatrix}
\right] = \left[\begin{smallmatrix}
D_1 f & D_2 f\\
D_1 g & D_2 g\\
\end{smallmatrix}
\right]
\end{align*}
The expansion of each term follows the same pattern. For instance, for $D_2 g$ we have:
\begin{equation*}
D_2 g = \sum_{i = 0}^\omega h^i C^i \otimes D_2 g_i + \mathcal{O}(h^{\omega + 1})\nonumber
\end{equation*}

Considering all this, the product $\Pi_{A} \left(\leftidx{^A}{\phantom{.}}{_{A}}\right) D_y F$ has two differentiated symbol groupings: top row (corresponding to $\Delta Q$) and bottom row (corresponding to $\Delta p$) groups.

Top row groups are the easiest ones as they are the ones that remain unaffected by $\Pi_{A}$. These are of the form:
\begin{equation*}
\sum_{i = 0}^\omega h^{i} \left[ A C^{i} \otimes U_{i} \right] + \mathcal{O}(h^{\omega + 1})
\end{equation*}
where $U_{i}$ are linear combinations of derivatives of $g$ and $f$ evaluated at the initial condition. Bottom row groups show more variety. These are of the form:
\begin{equation*}
\sum_{\left\vert\alpha\right\vert + \left\vert\beta\right\vert = 0}^\omega h^{\left\vert\alpha\right\vert + \left\vert\beta\right\vert} \left[ C^{\alpha} A \prod_{i = 1}^{\dim \beta - 1} \left( C^{\beta_{(i)}} A^{-} C^{\beta_{(i+1)}} A\right) \otimes V_{\alpha, \beta} \right] + \mathcal{O}(h^{\omega + 1})
\end{equation*}
where $V_{\alpha, \beta}$ are terms involving $\vartriangle$ and derivatives of $f$, $g$ and $\phi$ evaluated at the initial condition. The main difference here is that bottom row terms can have $C$s to the left of the first $A$, as well as the possibility of having $C \mapsto A^{-} C A$ substitutions to its right.

Putting everything together, and keeping in mind that $\omega = \min(\lambda, \mu, \nu)$ the expansion can be brought to the form:
\begin{align*}
\Delta Q &= h^2 \sum_{i = 0}^{m-1} h^{i} \left(\sum_{\alpha} K_{Q, \alpha_i} \otimes L_{Q, \alpha_i}\right) \Delta \Lambda_1 + \mathcal{O}(h^{m+2} \left\Vert\Delta \Lambda_1\right\Vert)\\
\Delta p &= h^2 \sum_{i = 0}^{m-1} h^{i} \left(\sum_{\alpha} K_{p, \alpha_i} \otimes L_{p, \alpha_i}\right) \Delta \Lambda_1 + \mathcal{O}(h^{m+2} \left\Vert\Delta \Lambda_1\right\Vert)
\end{align*}
where each $L_{j,\alpha_i}$ is again a combination of products of the derivatives of $f$, $g$, $\phi$ with $\vartriangle$ evaluated at the initial condition, and $K_{j, \alpha_i}$ is a Runge-Kutta symbol combination of order $\vert\alpha_i\vert$ as in theorem \ref{thm:vanishing_symbol_combinations}. The difference between $K_{Q, \alpha_i}$ and $K_{p, \alpha_i}$ lies on the fact that $K_{Q, \alpha_i}$ cannot begin with $C^i$ and there cannot be $C \mapsto A^{-}CA$ substitutions between the initial $A$ and the first $\hat{A}$, while on $K_{p, \alpha_i}$ it happens. Applying the result of said theorem all these terms vanish, which is what we set to prove.
\end{proof}

\begin{proposition}\label{prop:R_string_elimination}
In the Taylor expansion of $P_{A} D_z F_1$ only the terms belonging to the classes with elementary $R$-strings with a trailing zero, i.e. $\gamma$ $R$-strings of $\dim \gamma = s$ such that $\gamma_{(s)} = 0$, survive.
\end{proposition}
\begin{proof}
Given a class with an elementary representative $\gamma$ such that $\gamma_{(s)} \neq 0$ implies that it admits right appending, which gives us $\gamma'$. $R_{\gamma} = R_{\gamma'}$ by definition of class. On the other hand $S_{\gamma} \neq S_{\gamma'}$ and $\dim \gamma' = \dim \gamma + 2$, which means both terms will have opposite signs. Now $S_{\gamma'} = S_{\gamma} D_3 \tilde{g}_0 \vartriangle D_2 \tilde{\phi}_0$, but $S_{\gamma'} D_3 \tilde{g}_0 = S_{\gamma} D_3 \tilde{g}_0 \vartriangle \triangledown = S_{\gamma} D_3 \tilde{g}_0$, which is exactly what we needed to show that they cancel each other out. This is also true for other elements derived from the same elementary $R$-string via splitting and insertion, as they still necessarily admit right appending.
\end{proof}

\begin{theorem}
\label{thm:vanishing_symbol_combinations}
Assume an $s$-stage symplectic partitioned Runge-Kutta me\-thod with coefficients $A$ satisfying hypotheses \ref{itm:H1}, \ref{itm:H2}, \ref{itm:H3} (and consequently $\hat{A}$ satisfying \ref{itm:H1'} and \ref{itm:H2'}), together with conditions $D(r)$, $\hat{D}(q)$, $D\hat{D}(p - r)$ and $\hat{D}D(p - q)$. With $\alpha \geq 0$, we have:
\begin{equation}
e_s^T C^{\alpha} A \left(\prod_{i = 1}^{k} M_i \right) (\hat{A}_1 - \hat{A} A^{-} A_1) = 0, \quad 0 \leq k \leq \min(r, q, p - r, p - q) - 1
\label{eq:vanishing_combination_1}
\end{equation}
and:
\begin{equation}
e_s^T C^{\alpha} A \left(\prod_{i = 1}^{k} N_i \right) C A^{-} A_1, \quad 0 \leq k \leq \min(r, q, p - r, p - q) - 1
\label{eq:vanishing_combination_2}
\end{equation}
where $M_i$ and $N_i$ can be $C$, $A$, $\hat{A}$, $A^{-} C A$, $A C A^{-}$ for any $i$ except $k$ where $M_k = A C A^{-}$ cannot occur.
\end{theorem}

\begin{proof}
Multiplying $D(r)$ by $A^{-}$ we may obtain that:
\begin{align}
b C^k A^{-} &= e_s^T - k b C^{k-1}, \quad 1 \leq k \leq r
\label{eq:inv_D_assumption}
\end{align}

As $A$ satisfies \ref{itm:H3}, we also have that $e_s^T A = b$, and consequently $b A^{-} = e_s^T$.

The vanishing of the different symbol terms rests in both the vanishing of the following reduced combinations and the fact that any symbol combination that appears in the expansion can be brought to one of these.
\begin{itemize}
\item Combination 1:
\begin{equation*}
b C^{k-1} (\hat{A}_1 - \hat{A} A^{-} A_1) = 0, \quad 1 \leq k \leq \min(r,\hat{r})
\end{equation*}
This is said to be of order $k-1$, as that is the number of times $C$ appears. It vanishes because:
\begin{align*}
b C^{k-1} \hat{A}_1 &= k^{-1} b_1\\
b C^{k-1} \hat{A} A^{-} A_1 &= k^{-1} b (1 - C^{k}) A^{-} A_1\\
&= k^{-1} b A^{-} A_1  - k^{-1} b C^{k} A^{-} A_1\\
&= k^{-1} b_1  - k^{-1} \left(b_1 - k b C^{k-1} A_1\right)\\
&= k^{-1} b_1
\end{align*}
The application of the simplifying assumption $\hat{D}(\hat{r})$ in the second line and $D(r)$ in the fourth line are the limiting factors.

\item Combination 2:
\begin{equation*}
b C^{k} A^{-} A_1 = 0, \quad 1 \leq k \leq r
\end{equation*}
This is said to be of order $k$, as that is the number of times $C$ appears.
\begin{align*}
b C^{k} A^{-} A_1 &= b_1 - k b C^{k-1} A_1\\
&= b_1 - b_1\\
&= 0
\end{align*}
Again the application of the simplifying assumption $D(r)$ in the first line is the limiting factor.
\end{itemize}

Combination 1 and combination 2 can be generalized to the form \eqref{eq:vanishing_combination_1} and \eqref{eq:vanishing_combination_2} respectively.

As $c_s = 1$, we have that $e_s^T C^\alpha = e_s^T$, thus the $C^{\alpha}$ is there only for generality. After recursive application of $D$, $\hat{D}$, $D\hat{D}$, $\hat{D}D$ and \eqref{eq:inv_D_assumption} shows that each of these expressions can be brought to a linear combination of one of the reduced combinations with different values of $k$, which proves the theorem.
\end{proof}

\begin{remark}
For an $s$-stage Lobatto III A-B method we have that $s - 2 = r = p - q = q - 2 = p - r - 2$, thus:
\begin{equation}
e_s^T C^{\alpha} A \left(\prod_{i = 1}^{k} M_i \right) (\hat{A}_1 - \hat{A} A^{-} A_1) = 0, \quad 0 \leq k \leq s - 3
\label{eq:vanishing_combination_1_Lobatto}
\end{equation}
\begin{equation}
e_s^T C^{\alpha} A \left(\prod_{i = 1}^{k} N_i \right) C A^{-} A_1, \quad 0 \leq k \leq s - 3
\label{eq:vanishing_combination_2_Lobatto}
\end{equation}
\end{remark}

\begin{theorem}
\label{thm:local_error}
Assume an $s$-stage symplectic partitioned Runge-Kutta method with coefficients $A$ satisfying hypotheses \ref{itm:H1}, \ref{itm:H2}, \ref{itm:H3} (and consequently $\hat{A}$ satisfying \ref{itm:H1'} and \ref{itm:H2'}), together with conditions $B(p)$, $C(q)$, $D(r)$ (and consequently $\hat{B}(p)$, $\hat{C}(r)$, $\hat{D}(q)$). Then we have:
\begin{subequations}
	\label{eq:local_error_components}
	\begin{alignat}{1}
		\delta q_h (x) &= \mathcal{O}(h^{\min(p, q + r + 1) + 1}),\\
		\delta p_h (x) &= \mathcal{O}(h^{\min(p, 2 q, q + r) + 1}),\\
		\delta z_h (x) &= \mathcal{O}(h^{q}).
	\end{alignat}
\end{subequations}
\end{theorem}

\begin{proof}
The proof of this theorem is similar to that of \cite{Jay93}, theorem 5.1, which follows that of \cite{HaLuRo89}, theorem 5.9, and \cite{HaWa96}, theorem 8.10. (As the author was not initially used to working with \emph{trees}, we recommend a first look of at \cite{HaNoWa93}, theorem 7.4, for the non-initiated.)
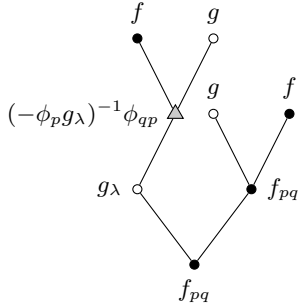
\begin{figure}[!h]
  \begin{minipage}[c]{0.35\textwidth}
	\begin{tikzpicture}[font=\footnotesize]
	  \tikzset{
	    level 1/.style={level distance=-10mm,sibling distance=15mm},
	    level 2/.style={level distance=-10mm,sibling distance=10mm},
	    level 3/.style={level distance=-10mm,sibling distance=10mm},
	    level 4/.style={level distance=-10mm,sibling distance=10mm},
	  }
	
	  \node[solid node,label=below:{$f_{p q}$}]{}
	    child{node(01)[hollow node,label=left:{$g_{\lambda}$}]{}
	      child{node(l1)[]{}edge from parent[draw=none]{}}
	      child{node(l2)[triangle node,label=left:{$(-\phi_p g_{\lambda})^{-1}\phi_{q p}$}]{}
	        child{node[solid node,label=above:{$f$}]{}edge from parent node[left]{}}
	        child{node[hollow node,label=above:{$g$}]{}edge from parent node[right]{}}
	        edge from parent node[right]{}
	      }
	      edge from parent node[right]{}
	    }
	    child{node(02)[solid node,label=right:{$f_{p q}$}]{}
	      child{node(l1)[hollow node,label=above:{$g$}]{}edge from parent node[left]{}}
	      child{node(l2)[solid node,label=above:{$f$}]{}edge from parent node[right]{}}
	      edge from parent node[left,xshift=0]{}
	    }
	  ;
	\end{tikzpicture}
  \end{minipage}\hfill
  \begin{minipage}[c]{0.65\textwidth}
    \caption{This order 6 tree represents the term $f_{p q}\left(g_{\lambda} (-\phi_p g_{\lambda})^{-1}\phi_{q p}(f,g), f_{p q}(g,f)\right)$. Note that the order is derived from the number of round nodes minus the number of triangle nodes. The tree itself can be written as $\left[[[\tau_Q,\tau_P]_{\lambda}]_P,[\tau_P,\tau_Q]_Q\right]_Q$ and corresponds to the Runge-Kutta term: $b_i \hat{a}_{i j} a^{-}_{j k} c_k^2 a_{i l} c_l^2$, where $a^{-}_{i j}$ are the components of the $A^{-}$ matrix.
    } \label{fig:sample_tree}
  \end{minipage}
\end{figure}

The arguments are essentially the same as those used in \cite{HaLuRo89} for $A$ invertible, but using a bi-colored tree extension (see fig.\ref{fig:sample_tree}). The inverses that appear need only be swapped by $A^{-}$. In these results two trees are used, $t$ and $u$ trees, referring to $y$ and $z$ equations respectively. In our case we will have both $t_Q$ and $t_P$ for $Q$ and $P$ equations, plus $u$ for $\lambda$ equations.

The key difference with respect to both this and Jay is that instead of only needing to set the limit such that for $[t,u]_y$ either $t$ or $u$ are above the maximum reduction order by $C(q)$ ($q + 1$ and $q - 1$), which leads to $2 q$, we need to be careful because we have two types of trees with $C(q)$ and $\hat{C}(r)$. First of all it is impossible to have $[t_Q, u]_Q$ as $f$ does not depend on $\lambda$, and we can only have $[t_Q, t_P]_Q$ which pushes the limit to $q + r + 2$. On the other hand, $[t_P,u]_P$ also sets a limit, which as it turns out is $q + r$. For both there is also the limit $q + r + 1$ set by $D(r)$, which trumps the limit set for $Q$ equations but it is trumped for $P$ equations by the one we just set.
\end{proof}

\begin{theorem}
\label{thm:global_error}
Consider the IVP posed by the partitioned differential-algebraic system of eqs.\eqref{eq:partitioned_system}, together with consistent initial values and the Runge-Kutta method \eqref{eq:RK_partitioned_system}. In addition to the hypotheses of theorem \ref{thm:local_error}, suppose that $\left\Vert R_A(\infty)\right\Vert \leq 1$  and $q \geq 1$ if $R_{A}(\infty)$. Then for $t_n - t_0 = n h \leq C$, where $C$ is some constant, the global error satisfies:
\begin{subequations}
	\label{eq:global_error_components}
	\begin{alignat}{1}
		q_n - q(t_n) &= \mathcal{O}(h^{\min(p, q + r + 1)})\label{eq:global_error_q}\\
		p_n - p(t_n) &= \mathcal{O}(h^{\min(p, 2 q, q + r)})\label{eq:global_error_p}\\
		z_n - z(t_n) &= \left\lbrace
			\begin{array}{rl}
				\mathcal{O}(h^{q}) & \text{if } -1 \leq R_A(\infty) < 1,\\
				\mathcal{O}(h^{q - 1}) & \text{if } R_A(\infty) = 1.
			\end{array} \right.\label{eq:global_error_lambda}
	\end{alignat}
\end{subequations}
\end{theorem}

\begin{proof}
Following the steps of \cite{Jay93}, theorem 5.2, for $\left\Vert R(\infty)\right\Vert < 1$ and $\left\Vert R(\infty)\right\Vert = 1$, $z_n - z(t_n)$ can be found to be of order $\mathcal{O}(h^{q})$ and $\mathcal{O}(h^{q - 1})$ respectively. As stated there, the result for $R(\infty) = -1$ can actually be improved to $\mathcal{O}(h^{q})$ by considering a perturbed asymptotic expansion.

Now, we proceed as in \cite{HaWa96}, theorem VI.7.5, applying \eqref{eq:RK_total_error_Q}\eqref{eq:RK_total_error_P}\eqref{eq:RK_total_error_Lam} to two neighbouring Runge-Kutta solutions, $\left\lbrace \tilde{q}_n, \tilde{p}_n, \tilde{\lambda}_n\right\rbrace$ and $\left\lbrace \hat{q}_n, \hat{p}_n, \hat{\lambda}_n\right\rbrace$, with $\delta_i = 0$, $\theta = 0$. Using the notation $\Delta x_n = \tilde{x}_n - \hat{x}_n$, we can write:
\begin{align*}
\Delta q_{n+1} &= \Delta q_n + \mathcal{O}\left( h \left\Vert \Delta p_n \right\Vert + h^{m + 2} \left\Vert \Delta \lambda_n \right\Vert\right)\\
\Delta p_{n+1} &= \Pi_{1,n} \Delta q_n + \Pi_{2,n} \Delta p_n + \mathcal{O}\left( h^{m + 2} \left\Vert \Delta \lambda_n \right\Vert\right)\\
\Delta \lambda_{n+1} &= \mathcal{R}_A(\infty) \Delta \lambda_n + \mathcal{O}\left( \left\Vert \Delta q_n \right\Vert + \left\Vert \Delta p_n \right\Vert + h \left\Vert \Delta \lambda_n \right\Vert \right)
\end{align*}
where $\Pi_{1,n}$ and $\Pi_{2,n}$ are the projectors defined in the statement of theorem \ref{thm:RK_total_error}, evaluated at $\hat{q}_n$, $\hat{p}_n$, $\hat{\lambda}_n$, and $m = \min(q - 1, r, p - q, p - r)$ for $-1 \leq R(\infty) < 1$ or $m = \min(q - 2, r, p - q, p - r)$ for $R(\infty) = 1$.

We can follow the same philosophy of \cite{HaLuRo89}, lemma 4.5, and try to relate $\left\lbrace\Delta q_{n},\Delta p_{n},\Delta \lambda_{n}\right\rbrace$ with $\left\lbrace\Delta q_{0},\Delta p_{0},\Delta \lambda_{0}\right\rbrace$. For this we make use of the fact that $\Pi_{i,n+1} = \Pi_{i,n} + \mathcal{O}(h)$, $\left(\Pi_{2,k}\right)^2 = \Pi_{2,k}$ and $\Pi_{2,k} \Pi_{1,k} = 0$ (these latter facts can be readily derived from their definition).

This leads to:
\begin{align*}
\left\Vert \Pi_{1,n+1} \Delta q_{n+1} \right\Vert &= \left\Vert \Pi_{1,n} \Delta q_n \right\Vert + \mathcal{O}\left( h \left\Vert \Delta p_n \right\Vert + h^{m + 2} \left\Vert \Delta \lambda_n \right\Vert\right)\\
\left\Vert \Pi_{2,n+1} \Delta p_{n+1} \right\Vert &= \left\Vert \Pi_{2,n} \Delta p_n \right\Vert + \mathcal{O}\left( h\left\Vert \Delta q_n \right\Vert + h^{m + 2} \left\Vert \Delta \lambda_n \right\Vert\right)\\
\left\Vert \mathcal{R}_A(\infty) \Delta \lambda_{n+1} \right\Vert &= \left\Vert \mathcal{R}_A(\infty) \right\Vert^2 \left\Vert \Delta \lambda_{n} \right\Vert + \mathcal{O}\left( \left\Vert \Delta q_n \right\Vert + \left\Vert \Delta p_n \right\Vert + h \left\Vert \Delta \lambda_n \right\Vert \right)
\end{align*}

Thus the error estimates become:
\begin{align*}
\left\Vert \Delta q_{n} \right\Vert &\leq C_q \left( \left\Vert \Delta q_0 \right\Vert + h \left\Vert \Delta p_0 \right\Vert + h^{m + 2} \left\Vert \Delta \lambda_0 \right\Vert\right)\\
\left\Vert \Delta p_{n} \right\Vert &\leq C_p \left( \left\Vert \Pi_{1,0} \Delta q_0 \right\Vert + \left\Vert \Pi_{2,0} \Delta p_0 \right\Vert + h^{m + 2} \left\Vert \Delta \lambda_0 \right\Vert\right)\\
\left\Vert \Delta \lambda_{n} \right\Vert &\leq C_{\lambda} \left( \left\Vert \mathcal{R}_A(\infty) \right\Vert^n \left\Vert \Delta \lambda_{0} \right\Vert + \left\Vert \Delta q_0 \right\Vert + \left\Vert \Delta p_0 \right\Vert + h \left\Vert \Delta \lambda_0 \right\Vert \right)
\end{align*}

Proceeding as in \cite{HaLuRo89} to use the Lady Windermere's Fan construction and using the results from theorem \ref{thm:local_error} for $\delta q_h(t_k), \delta p_h(t_k)$, and the results we derived for $\delta \lambda_h(t_k)$, with $m = \min(q - 1, r, p - q, p - r)$ for $-1 \leq R(\infty) < 1$ as well as $m = \min(q - 2, r, p - q, p - r)$ for $R(\infty) = 1$, we find the global error by addition of local errors, which gives the result we were looking for.
\end{proof}

\begin{corollary}
\label{cor:global_error_Lobatto}
The global error for the Lobatto IIIA-B method applied to the IVP posed by the partitioned differential-algebraic system of eqs.\eqref{eq:partitioned_system} is:
\begin{subequations}
	\label{eq:Lobatto_error_components}
	\begin{alignat}{1}
		q_n - q(t_n) &= \mathcal{O}(h^{\min(2 s - 2)}),\label{eq:Lobatto_error_q}\\
		p_n - p(t_n) &= \mathcal{O}(h^{\min(2 s - 2)}),\label{eq:Lobatto_error_p}\\
		z_n - z(t_n) &= \left\lbrace
			\begin{array}{rl}
				\mathcal{O}(h^{s}) & \text{if $s$ even},\\
				\mathcal{O}(h^{s - 1}) & \text{if $s$ odd}.
			\end{array} \right.\label{eq:Lobatto_error_lambda}
	\end{alignat}
\end{subequations}
\end{corollary}

\begin{proof}
To prove this it suffices to substitute $p = 2 s - 2$, $q = s$, $r = s - 2$ and $\mathcal{R}_A(\infty) = (-1)^{s - 1}$ in the former theorem.
\end{proof}

\section{Conclusion}
In this paper we have proposed a new numerical scheme for partitioned index 2 DAEs proving its order.  The method opens the possibility to construct high-order methods for nonholonomic systems in a systematic way, preserving the nonholonomic constraints exactly. So far, the methods to numerically integrate a given nonholonomic system were constructed using discrete gradient techniques or modifications of variational integrators based on discrete versions of the Lagrange-d'Alembert's principle. Integrators in the latter category, in which our method falls, tend to display a certain amount of arbitrariness or awkwardness, particularly in the way constraints are discretized or imposed. In most cases, with the exception of SPARK methods \cite{Jay09}, the resulting methods are limited to low order unless composition is applied, and without a general framework for error analysis. However, our method offers a clear and natural way to construct them to arbitrary order. Further considerations about our construction, particularly with respect to its interpretation will be left for \cite{NonholonomicMartinSato18}.

\section*{Acknowledgements}
The author has been partially supported by Ministerio de Ciencia e Innovaci\'on  (MICINN, Spain) under grants MTM 2013-42870-P, MTM 2015-64166-C2-2P, MTM2016-76702-P and ``Severo Ochoa Programme for Centres of Excellence'' in R\&D (SEV-2015-0554). The author thanks MICINN for an FPI grant. Thanks also David Mart{\'\i}n de Diego for helpful comments and the Geometry, Mechanics and Control Network fot its support.

\bibliography{bib_NonholonomicNumAnalysis}
\bibliographystyle{plain}

\end{document}